%% file: FIV.tex
\documentclass[11 pt]{article}
\input{commands}

\title{\textsc{Free Products in R. Thompson's Group $V$}}

\author{Collin Bleak \and Olga Salazar-D\'iaz} 

\begin{document}

\maketitle

\abstract{We investigate free product structures in R. Thompson's
  group $V$, primarily by studying the topological dynamics associated
  with $V$'s action on the Cantor Set.  We show that the class of free
  products which can be embedded into $V$ includes the free product of
  any two finite groups, the free product of any finite group with
  $Q/Z$, and the countable non-abelian free groups.  We also show the
  somewhat surprising result that $Z^2*Z$ does not embed in $V$, even
  though $V$ has many embedded copies of $Z^2$ and has many embedded
  copies of free products of pairs of its subgroups.}

\setcounter{page}{1} \pagestyle{plain}
\thispagestyle{empty} \input{fiv.body}
\newpage
\bibliographystyle{amsplain}
\bibliography{dirBib}
\end{document}

%% file: commands.tex
\usepackage{latexsym}
\usepackage{amsthm}
\usepackage{amssymb}
\usepackage{amsmath}
\usepackage{amsfonts}
\usepackage{mathrsfs}
\usepackage[all]{xy}
\usepackage{texdraw}
\usepackage{graphicx}
\usepackage{psfrag}
\usepackage{makeidx}

\newtheorem{Theory}{Theory}[section] 
\newtheorem{theorem}[Theory]{Theorem}
\newtheorem{lemma}[Theory]{Lemma}
\newtheorem{technicalLemma}[Theory]{Technical Lemma}
\newtheorem{corollary}[Theory]{Corollary}
\newtheorem{proposition}[Theory]{Proposition}
\newtheorem{fact}{Fact}  
\newtheorem{remark}[Theory]{Remark} 
\newtheorem{question}{Question} 
\newtheorem{conjecture}[question]{Conjecture}

\newtheorem{Ntn}{Description} 
\newtheorem{Dn}[Ntn]{Definition}

\newcommand{\be}{\begin{enumerate}}
\newcommand{\ee}{\end{enumerate}}
\newcommand{\bq}{\begin{question}}
\newcommand{\eq}{\end{question}}
\newcommand{\bcj}{\begin{conjecture}}
\newcommand{\ecj}{\end{conjecture}}
\newcommand{\bc}{\begin{corollary}}
\newcommand{\ec}{\end{corollary}}
\newcommand{\bl}{\begin{lemma}}
\newcommand{\el}{\end{lemma}}
\newcommand{\btl}{\begin{technicalLemma}}
\newcommand{\etl}{\end{technicalLemma}}
\newcommand{\bt}{\begin{theorem}}
\newcommand{\et}{\end{theorem}}
\newcommand{\bp}{\begin{proposition}}
\newcommand{\ep}{\end{proposition}}
\newcommand{\bft}{\begin{fact}}
\newcommand{\eft}{\end{fact}}
\newcommand{\brk}{\begin{remark}}
\newcommand{\erk}{\end{remark}}
\newcommand{\bd}{\begin{Dn}}
\newcommand{\ed}{\end{Dn}}

\author{Collin Bleak}

%% file: fiv.body.tex
\newcommand{\supp}[1]{\textrm{Supp(}#1\textrm{)}}
\newcommand{\csupp}[1]{\overline{\textrm{Supp(}#1\textrm{)}}}
\newcommand{\Imp}[1]{\textrm{I(}#1\textrm{)}}

\section{Introduction} 

We prove some results related to the subgroup structure of
R. Thompson's group $V$.  In particular, we explore conditions on
factor groups so that free products of non-trivial factors can embed
into $V$.  Our first result shows that $V$ contains many free products
of various of the isomorphism classes of its non-trivial subgroups.
Our second theorem states that (although $V$ contains many free
products as above and many copies of $Z^2$) the group $Z^2*Z$ does not
embed into $V$.  Of particular interest in this exploration is that
the non-embedding result seems very difficult to prove using algebraic
methods (say, using a presentation of $V$).  The present authors use
topological dynamics (via the characterization of $V$ as a group of
homeomorphisms of the Cantor Set $\mathfrak{C}$) to attain these
results.

Let $\mathcal{FPV}$ denote the class of groups which 
\begin{enumerate}
\item admit decompositions as free products of pairs of non-trivial
  subgroups, and
\item embed into $V$.
\end{enumerate}

Also, let $\mathcal{A}$ denote the
smallest class of groups so that 
\begin{enumerate}
\item $\mathcal{A}$ contains all finite groups,
\item $Z\in \mathcal{A}$,
\item $Q/Z\in\mathcal{A}$, and 
\item $\mathcal{A}$ is closed under 
\begin{enumerate}
\item isomorphism,
\item passing to subgroup, and
\item taking the direct product of any finite member with any member.
\end{enumerate}
\end{enumerate}

We now state our main results.

\begin{theorem}
\label{embedding}

If $K_1$, $K_2\in \mathcal{A}$ are non-trivial groups, then the group
$K_1*K_2\in \mathcal{FPV}$.

\end{theorem}
Some previously observed consequences of the above theorem are
firstly, that $V$ contains embedded copies of all of the countable
non-abelian free groups (an obvious fact), and secondly, that $V$
contains embedded copies of PSL($2,Z$) $\cong Z_2 * Z_3$ ($T\leq V$,
and recall that $T$ is $C^0$ conjugate to the group of PSL($2$,$Z$)
homeomorphisms of $RP^1$ with rational breaks in slope (see
\cite{GhysSergiescu, TTransvections})).

On the other hand, the following theorem shows that while free
products of groups in the isomorphism classes of $V$'s subgroups are often
available in $V$, one cannot choose these subgroups indiscriminately.
\begin{theorem} 
\label{nonembedding}

The group $Z^2*Z$ does not embed in $\mathcal{FPV}$.

\end{theorem}

The original motivation for the work in this paper sprang from the
question ``Does $Z^2 * Z$ embed in Thompson's group $V$?'', which
was asked of the first author by Mark Sapir.  Some context is given
below.

The team of Holt, R\"over, Rees and Thomas introduce and analyze the
class of groups which have context free co-word problem in
\cite{HRRTcfc}.  They show this class of groups is nice in various
ways. For instance, it is closed under direct products, standard
restricted wreath products where the top group is in the subclass of
groups with context-free word problem, passing to finitely generated
subgroups, and passing to finite index overgroups.  They further
conjecture that the class is not closed under free product, and
currently, one of the lead candidates for proving this conjecture is
$(Z^2)*Z$.

In the paper \cite{LehnertSchweitzer}, Lehnert and Schweitzer show
that R. Thompson's group $V$ is a CFC group.  In particular, if $Z^2 *
Z$ embeds in $V$, then it too would be a CFC group.

Thus, our chief result, in terms of this thread of the development of
the theory of CFC groups, simply says that $Z^2 * Z$ remains a
reasonable candidate for proving the conjecture that the class of CFC
groups is not closed under free product.

There is another non-embedding result for $V$ which is of interest in
this context.  Higman in \cite{HigmanSimplicity} uses his
\emph{semi-normal forms} to study the dynamics of automorphism groups
$G_{n,r}$ acting on specific algebras (where $n\geq 2$ and $r$ are
positive integers).  Semi-normal forms can help detect infinite orbits
under these actions, and other nice properties of elements of the
groups $G_{n,r}$.  In any case, $V = G_{2,1}$ in Higman's notation,
and Higman shows in \cite{HigmanSimplicity}, using semi-normal forms,
that GL($3,Z$) does not embed into $G_{n,r}$ for any indices $n$ and
$r$.  (Brin's revealing pair technology has many significant parallels
with Higman's semi-normal forms, though it was developed
independently.)  

There is a technical dynamical property $(*)$ for subgroups of $V$,
and a set $\mathcal{D}$ of subgroups of $V$ which satisfy $(*)$, which
we call the \emph{demonstrative groups}.  These groups are easy to
find as factor subgroups of free product subgroups in $V$; property
$(*)$ is useful when attempting to build Ping-Pong constructions (we
will briefly discuss Fricke and Klein's classical ``Ping-Pong Lemma''
in the next section).

While we do not know that every group isomorphic to a group in
$\mathcal{D}$ can be found in $\mathcal{A}$, it is not too
hard to show that every group in $\mathcal{A}$ can embed in $V$ as a
group in $\mathcal{D}$.  Thus, Theorem \ref{embedding} is actually a
corollary of the following lemma.

\begin{lemma}
\label{DInFFPV}

If $K_1$, $K_2\in \mathcal{D}$ are non-trivial groups then $K_1*K_2\in
\mathcal{FPV}$.

\end{lemma}

One reasonable place to try to extend the class $\mathcal{A}$ is to
replace property (4.c) with closure under the general extension of a
finite member by any member, creating a class $\mathcal{B}$.  In this
case $\mathcal{B}$ would contain all of the virtually cyclic groups.
It may be easy to prove that $\mathcal{D}$ contains isomorphic copies
of the virtually cyclic groups; the present authors have made no
efforts in that direction.

  We state the following questions.

\begin{question}
Is it true that if $G$ and $H$ are non-trivial subgroups of $V$, with
$\mid G\mid\geq 3$ and $\langle G,H\rangle\cong G*H$ in $V$, then
there are distinct non-empty sets $P_G$ and $P_H$ in $\mathfrak{C}$ so that for
any non-trivial elements $g\in G$ and $h\in H$ we have $P_H g\subset
P_G$ and $P_G h\subset P_H$?
\end{question}

Colloquially, must every free product of groups in $V$ arise from a
Ping-Pong in $V$?

Let $\mathcal{CFPV}$ be the smallest class of groups which contains
$\mathcal{B}$ and which is closed under
\begin{enumerate}
\item isomorphism,
\item passing to subgroups,
\item extending any finite member by any member, and
\item taking free products of any two of its members.
\end{enumerate}
\begin{question}
Does $\mathcal{FPV} = \mathcal{CFPV}$?
\end{question}

On one final note, we should like to mention the Java software package
which Roman Kogan of SUNY Stony Brook wrote at the NSF funded Cornell
Summer 2008 Mathematics Research Experience for Undergraduates.  His
software provides a convenient interface for calculating and storing
products, inverses, and conjugation in the higher dimensional Thompson
groups $nV$ created by Brin in \cite{BrinHigherV} (and thus, in the
R. Thompson groups $F<T<V$ as well).  While we did not use the
software to prove anything in this paper, we often used it during
exploration to verify that our initial constructions worked as intended.

The authors would like to thank Daniel Farley for his participation in
early phases of this project.  The first author would also like to
thank Mark Sapir for interesting discussions with regards to the
family of R. Thompson groups.  This note arose as a consequence of one
such discussion.

\section{Basic tools}
In this section, we define the terminology of the paper, and state
some easy or known facts that will be useful for us in our arguments.

\subsection{The Cantor Set}
We define language and notation describing the Cantor
Set, and various aspects and subsets thereof.

We let $\mathcal{T}_2$ represent the infinite binary tree.  We label
the nodes of $\mathcal{T}_2$ by finite binary strings in the usual
fashion (thus, the node labelled $1101$ corresponds to the node found
by starting at the root of $\mathcal{T}_2$, and travelling down to the
right child twice, then to the left child, and finally to the right
child, see the diagram below).  We identify the Cantor Set
$\mathfrak{C}$ with the infinite descending paths from the root of
$\mathcal{T}_2$, as is standard.  We will say \emph{$x\in
  \mathfrak{C}$ underlies a node $n$ of $\mathcal{T}_2$} if $n$ is a
node of $\mathcal{T}_2$ and the path for $x$ passes through $n$.  We
will call the set of such points the \emph{Cantor Set underlying $n$}, and denote this set by $\mathfrak{C}_n$.

By way of example, the set of points in the Cantor Set underlying
``1101'' is given as the set of elements of $\{0,1\}^{N}$ which begin with
the string ``1101''.  These correspond to all of the infinite descending
paths on $\mathcal{T}_2$ which start at the root and pass through the
node ``1101'' labelled in the diagram below.

\begin{center}
\includegraphics[height=150pt,width=150 pt]{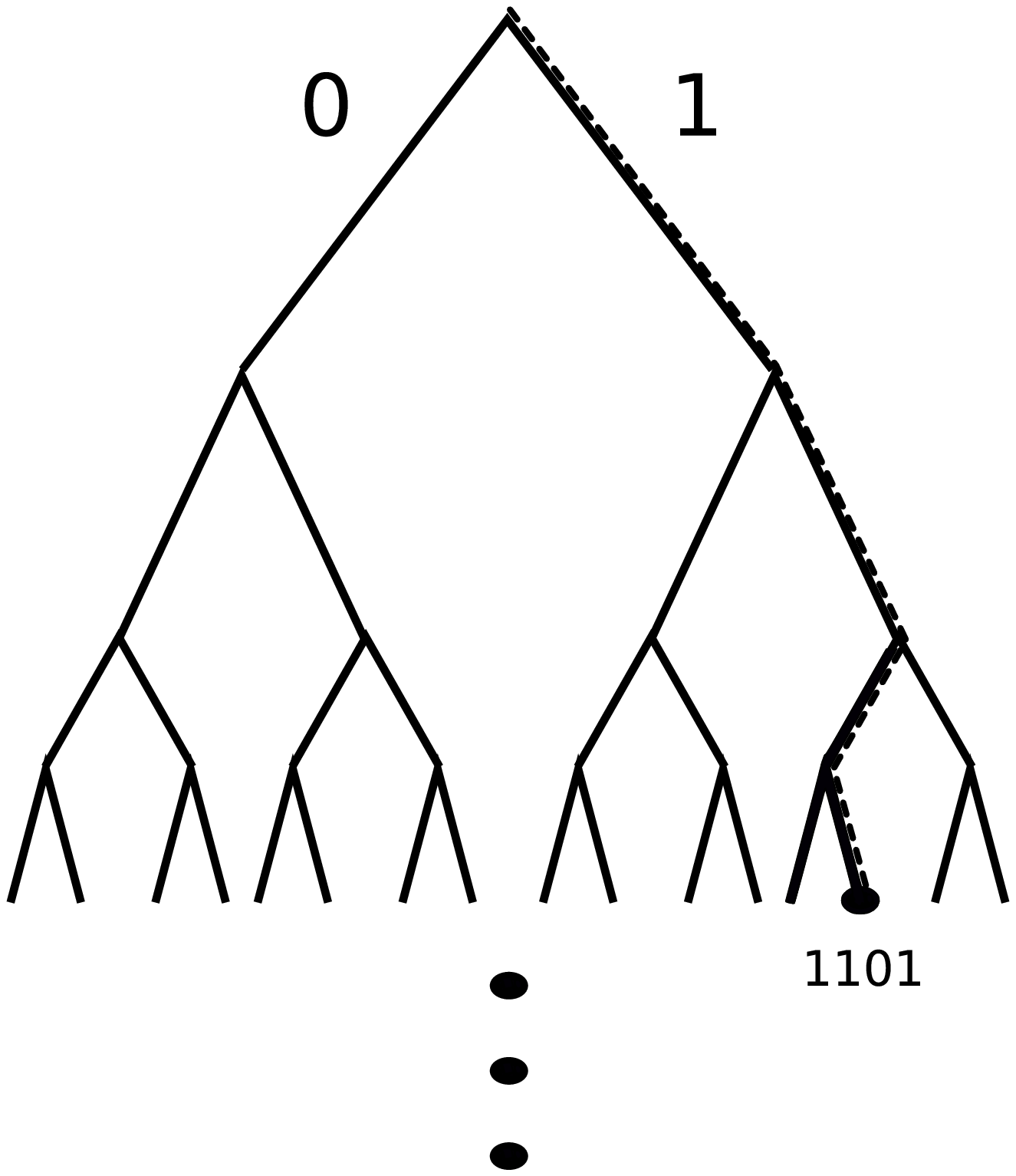}
\end{center}

Thus, for the example above, we have
\[
\mathfrak{C}_{1101} = \left\{(s_n)\in\mathfrak{C}\mid (s_n) = 1101s_4s_5s_6\ldots\right\}
\]

For a finite collection $S$ of points in $\mathfrak{C}$, we will
specify that \emph{a neighborhood $U$ of $S$} is precisely given as a
finite union of the Cantor Sets underlying a finite set $\mathcal{N}$
of nodes of $\mathcal{T}_2$, where we further require that for each
point $s\in S$, there is precisely one node $n_s\in \mathcal{N}$ so
that $s$ underlies the node $n_s$.  In particular, a neighborhood of
point $x$ in $\mathfrak{C}$ will be thought of as the Cantor Set
underlying some node $n$ of $\mathcal{T}_2$, where $x$ underlies $n$.
This specification of the usual notion of neighborhood should cause
the reader no confusion in the body of this paper.  Given a set
$S\subset\mathfrak{C}$, we will call any open subset
$U\subset\mathfrak{C}$ containing $S$ a \emph{general neighborhood of
  $S$}.  We shall require the use of a general neighborhood only one
time in this note.

\subsection{R. Thompson's group $V$}
The group $V$ is a specific collection of homeomorphisms of
$\mathfrak{C}$, under the operation of composition.  Each such element
$u$ of $V$ can be represented non-uniquely as a pair of finite binary
subtrees of $\mathcal{T}_2$, each with the same number $n$ of leaves,
together with an element $\sigma\in \Sigma_n$, the permutation group
on $n$ letters.  

We will write such a representation as $u \sim (D,R,\sigma)$, where $D$
and $R$ are our finite binary trees, and $\sigma$ is the permutation,
as mentioned above.  We will often call this simply a \emph{tree-pair}, and
drop explicit mention of the permutation unless we need it in the
course of events.  We may write $P = (D,R)$ in this case, considering
$P$ to be a tree pair (with permutation) which guides us as a rule
defining the element $u$ of $V$ it is intended to represent.  For the
remainder of this section, we will assume $u$ and $P$ are fixed for
the purpose of discussion, and to help us define terminology.

The rule which translates a tree-pair into a homeomorphism of the
Cantor Set is as follows.  Consider $D$ and $R$ as subtrees of
$\mathcal{T}_2$.  For each leaf $i$ of $D$ (where $i\in
\left\{1,2,\ldots ,n\right\}$), map the Cantor Set underlying $i$ in
the unique orientation-preserving, bijective, affine fashion to the
Cantor Set underlying the leaf $\sigma(i)$ of $R$.  Thus $D$
represents the domain, and $R$ represents the range.  

Thoughout this note, we will have elements of $V$ act on $\mathfrak{C}$
on the right, so that if $x\in \mathfrak{C}$, we will denote by $xv$
the image of $x$ under the action of $v$.  If $X\subset \mathfrak{C}$,
and $v\in V$, then we will denote by $Xv$ the set $\left\{xv|x\in
X\right\}$.  Following these conventions, if $u$ is also an element of
$V$, then the symbols $u^v$ and $[u,v]$ are defined by the equations
$u^v = v^{-1}uv$ and $[u,v] = u^{-1}v^{-1}uv = (v^{-1})^u\cdot
v=u^{-1}\cdot u^v$.

If $v\in V$, we define $\supp{v} = \{x\in\mathfrak{C}\mid xv\neq
x\}$.  We now have the following standard lemma from permutation group
theory, which we use freely in the remainder.

\begin{lemma}
\label{throwSupport}
Let $u,v\in V$, then $\supp{u^v}= \supp{u}v$.
\end{lemma}

Finally, given $v\in V$ and $x\in \mathfrak{C}$, we define the orbit
of $x$ under $v$ to be the set 
\[
\mathcal{O}\left(x,v\right) = \left\{xv^k\mid k\in Z\right\}.
\]
  If this set
does not have cardinality one, then we say that \emph{the orbit of
  $x$ under $v$ is non-trivial}.  Other language to this effect is to
be interpreted in the obvious fashion and should cause the reader no
confusion.

We note that given any $v\in V$, there is an induced ``action'' on
``most'' of the nodes of the infinite binary tree $\mathcal{T}_2$, in
the following sense.  If we pass to any leaf $l$ of a representative
tree pair for $v$, or any node $m$ below $l$ (lying further down some
infinite descending path from the root which passes through $l$), then
the node $l$ (or $m$) will be mapped to a node $n$ of $\mathcal{T}_2$;
that is, the Cantor set underlying $l$ (or $m$) will be carried
affinely and bijectively to the Cantor Set underlying $n$.  We thus
say $lv = n$ or $mv =n$ when we are thinking of the induced action of
$v$ on nodes of $\mathcal{T}_2$.  (Note, nodes above the
leaves of the domain tree will very likely get ``split'' by $v$, so
that there is always a finite subset of $\mathcal{T}$ where this sort
of induced action makes no sense.)

\subsection{Revealing pairs}
We will make use of Brin's revealing pair technology (see
\cite{BrinHigherV}) in order to define interesting subsets of the
Cantor Set $\mathfrak{C}$ for specific elements of $V$.  The general
argument can be given without using revealing pairs, but they provide
a useful context for our discussion.  We will define here everything
which is required for this note.  Lemmas and corollaries appearing in
this subsection before Lemma \ref{commonImportantPoints} can all
either be found in \cite{BrinHigherV} or in \cite{Salazar}, or are
easy consequences of the results found therein.

For the tree pair $P = (D,R)$, we can consider the common tree $C =
D\cap R$, which is the finite rooted subtree of $\mathcal{T}_2$ with
each node a node of both $D$ and $R$.  Each leaf of $C$ is either the
root of a descending subtree of $D$, or the root of a descending
subtree of $R$ (in either such case, we call these descending trees
\emph{components of} $D\backslash R$ or of $R\backslash D$, as the
case may be), or a leaf of both $D$ and $R$ (in this case, we call the
leaf a \emph{neutral leaf} of $P$).

The tree pair $P$ is called a \emph{revealing pair} if it satisfies
two conditions.  The first condition is that for each complementary
component $X$ of $D\backslash R$, $X$ has a leaf $r_X$ which, under
iteration of the rule $P$, travels through the neutral leaves of $C$
until it is finally mapped to the root of $X$ ($r_X$ is unique for $X$,
and is called the \emph{repelling leaf of $X$} or the \emph{repeller
  of $X$}).  The second condition is similar; if $n_Y$ is the root of a
component $Y$ of $R\backslash D$, then iteration of the rule $P$ has
$n_Y$ travel through the neutral leaves of $C$ until it finally maps
beneath itself to a leaf $l_Y$ of $Y$ (the leaf $l_y$ is called the
\emph{attracting leaf of $Y$} or the \emph{attractor of $Y$}).  By the
discussion preceeding Lemma 10.2 of \cite{BrinHigherV}, each element
in $V$ has a revealing pair representative (and as before, this is not
unique).

Let $v\in V$.  An easy consequence of Proposition 10.1 of
\cite{BrinHigherV} is that there is a minimal non-negative power $k$
so that $v^k$ acts on $\mathfrak{C}$ with no non-trivial finite
orbits.  Set $u = v^k$, so that $u$ admits no
non-trivial finite orbits in its action on $\mathfrak{C}$.  Finally,
let us assume $P = (D,R)$ is actually a revealing pair representing $u$.

We obtain a list of useful, obvious results, which we leave to the
reader to verify.  The phrase ``maps to'' in this following lemma is
referring to the action of $w$ on nodes of $\mathcal{T}_2$.

\begin{lemma}
\label{noOrbitsProperties}

Suppose $w\in V$ so that $w$ admits no non-trivial finite orbits in
its action on the Cantor Set $\mathfrak{C}$.  Suppose further that the
revealing pair $P_w = (D_w,R_w)$ represents $w$.

\begin{enumerate}
\item Any repeller $r_X$ of a component $X$ of $D_w\backslash R_w$ always
  maps to the root of $X$ by the rule $P_w$.
\item The root $n_Y$ of any component $Y$ of $R_w\backslash D_w$ always
  maps to the attractor $l_Y$ of $Y$ by the rule $P_w$.
\item The map $w$ restricted to any Cantor Set
  underlying a node $r_X$ or $n_Y$ as above is affine with
  slope not equal to one.
\item Every point in $\mathfrak{C}$ which is fixed by $w$ and which
  does not underly a node $r_X$ or $n_Y$ as above lies under a neutral
  leaf $n$ of the pair $P_w$ upon which $w$ must act as the identity.
\end{enumerate}
\end{lemma}

We continue our discussion with the element $u$ constructed previously.

By an application of the standard Contraction Lemma, we observe that
if a leaf $l$ of $D$ is mapped above or below itself in $\mathcal{T}_2$ by
the rule $P$, then there will be a unique fixed point in the Cantor
Set underlying $l$ (if $l$ maps above itself, consider the
inverse map $u^{-1}$ in order to force a contraction).  Fixed points
underlying repellers of $D$ will be called \emph{repelling fixed
  points of $u$}, and fixed points underlying attractors in components of
$D\backslash R$ will be called \emph{attracting fixed points of $u$}.

\begin{corollary}
\label{fixedPointCorollary}

Under the same hypothesies as in Lemma \ref{noOrbitsProperties}, for
each repeller $r_x$ there is a unique repelling fixed point $p_x$
underlying it, and for each attractor $l_Y$ there is a unique
attracting fixed point underlying it.

\end{corollary}

The diagram below illustrates a likely tree pair for such an element
$u$.  This particular tree-pair indicates that the element $u$ has one
repelling fixed point (under 0010) and two attracting fixed points (under 10 and 11 respectively).

\begin{center}
\includegraphics[height=150pt,width=250 pt]{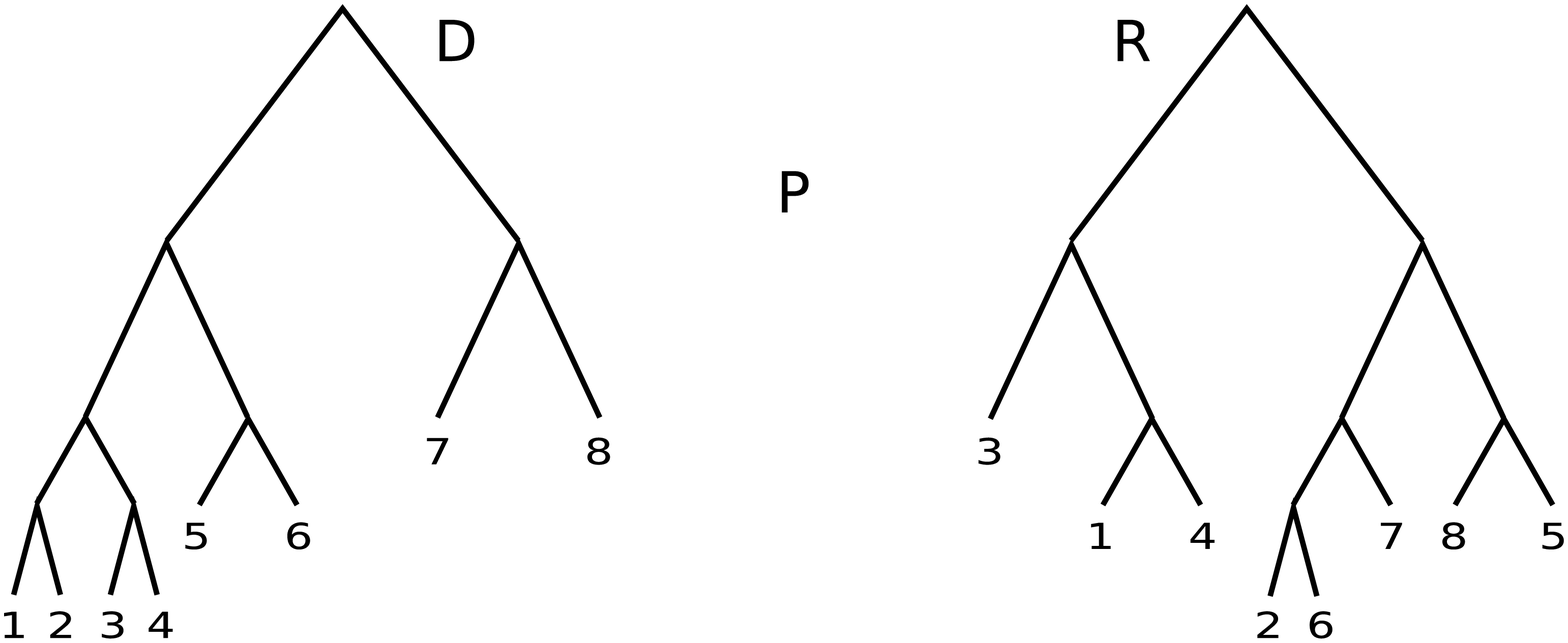}
\end{center}

We call the repelling and attracting fixed points under the action of
$u$ the \emph{important points of $u$}, and denote the set of such as
$I(u)$.

Now, throughout the remainder of this paper, if we discuss the
important points of an element $w$ of $V$, it is to be understood that
$w$ does not admit finite non-trivial orbits in its action on
$\mathfrak{C}$.  In this case, given a revealing pair $Q = (S,T)$
representing $w$, the Cantor Set underlying each root of a component
of $S\backslash T$ represents a \emph{repelling basin for $w$}, and
the Cantor Set underlying a root of a component of $T\backslash S$
represents an \emph{attracting basin for $w$}.

Let $\left\{E_i\right\}_{i = 1}^n$ represent the components of
$D\backslash R$ and let $\left\{F_j\right\}_{j = 1}^m$ represent the
components of $R\backslash D$.  For each $i$, $E_i$ contains a
repelling leaf, as defined above, and all other leaves of $E_i$ are
called \emph{sources}.  Likewise, for each component $F_j$, there is
an attracting leaf, as defined above, and all other leaves are called
\emph{sinks}.  For each source leaf $s_0$, there is a path $s_0 = n_0,
n_1,\ldots, n_t = s_k$ through neutral leaves $n_1, \ldots, n_{t-1}$
of $C$, and then visiting a sink $s_k$, so that $u^p$ will throw the
Cantor Set underlying $s_0= n_0$ onto the Cantor Set underlying $n_p$
for all indices $0\leq p\leq t$.  We call this path the
\emph{source-sink chain $s_0-s_k$ for $P$}.  We can now make a
bi-partite graph, whose vertices are labelled by the repelling and
attracting basins, and whose edges correspond with and are labelled by
source-sink chains connecting repelling basins to attracting basins.
We call this graph the \emph{flow graph for $P$}.  (In the presence of
finite non-trivial orbits, the flow graph has other information
appended to it as in \cite{REU2007, Barker}.  Here, we only
define the portions of the standard flow graph required to support the
discussion in the remainder of this note.)

For our example above, the flow graph is as diagrammed.

\begin{center}
\includegraphics[height=125pt,width=300 pt]{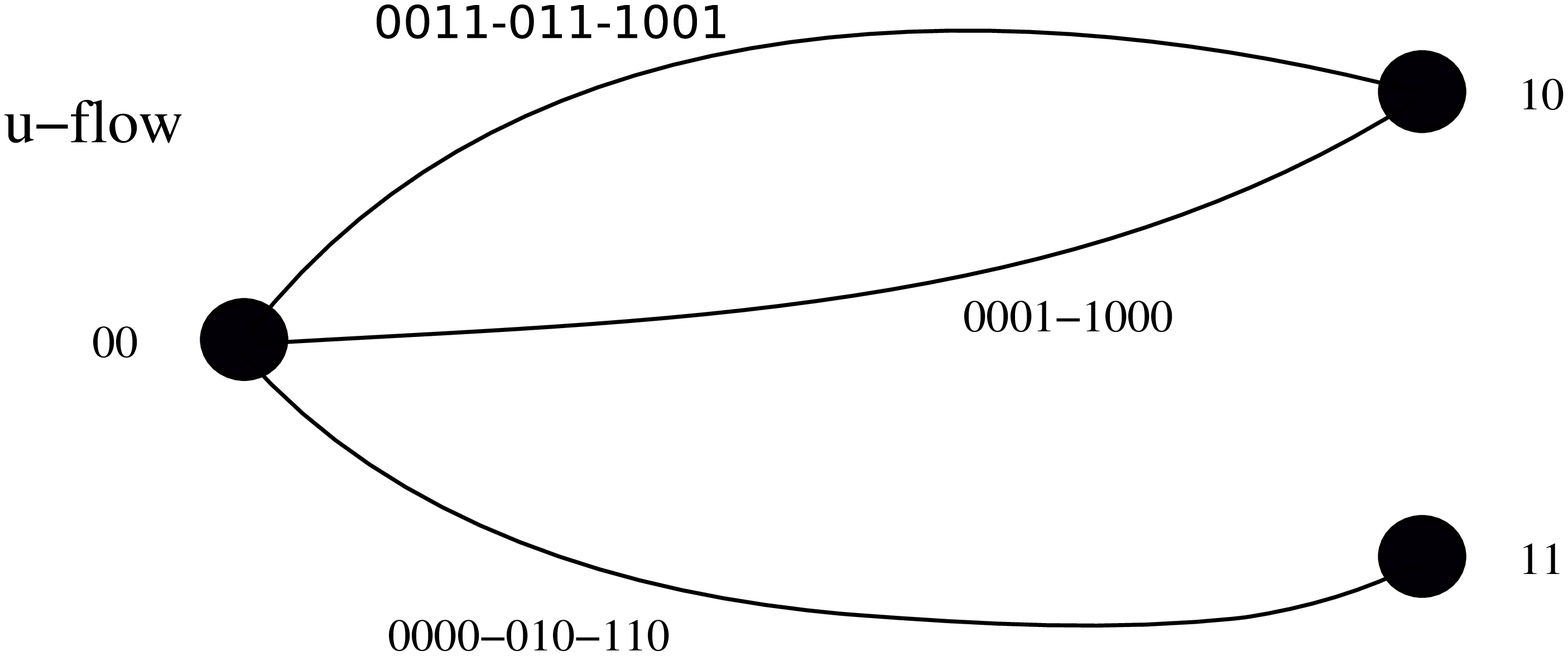}
\psfrag{u-flow}{$u$-flow}
\end{center}

Let $X$ now represent a connected component of the flow graph for $P$.
We can form a set, the \emph{Cantor Set underlying $X$}, by taking the
union of the Cantor Sets underlying the attracting and repelling
basins of $X$, and underlying the neutral leaf nodes occuring in the
source-sink chain labels of the edges of $X$.  This union is
immediately independent of the revealing pair representing $u$, so
that we will also call it a \emph{component of support of $u$}.  We
will define the \emph{component support of $u$}, denoted $\csupp{u}$, as the
union of the components of support of $u$.  We note that $\csupp{u}$
actually is the topological closure of the support of $u$, since it
consists of the support of $u$ together with the important points of
$u$.

In particular, we have a useful lemma, which is easy to see if the
above is understood.

\begin{lemma}
\label{commonImportantPoints}

Let $p\in\mathfrak{C}$ and $g$, $h\in V$ so that they admit no finite
non-trivial orbits and with $p\in I(g)\cap I(h)$.  Then, there
is a node $n$ in $\mathcal{T}_2$ so that 
\begin{enumerate}
\item $p$ underlies $n$, 
\item $\mathfrak{C}_n\subset \csupp{g}\cap\csupp{h}$, and
\item the commutator $[g,h]$ acts as the identity over
  $\mathfrak{C}_n$.
\end{enumerate}
\end{lemma}

We now give three background lemmas from \cite{REU2007}, including
their proofs for completeness.

\begin{lemma}
\label{importantPointPreservation}
Let $g$, $h\in V$ so that they admit no finite non-trivial orbits, and
so that $g$ and $h$ commute. We then have I($g$)$\cap$I($h$) =
I($g$)$\cap \csupp{h} =$ I($h$)$\cap\csupp{g}$.
\end{lemma}
\begin{proof}
We show that if $x\in$ I($g$)$\cap\csupp{h}$, then $x\in $ I($h$),
from which the lemma immediately follows.

If $x$ is not an important point of $h$, then $g$ must have infinitely
many important points, since the full orbit of $x$ under that action
of $h$ consists of important points for the functions $g^{(h^n)}=g$.
\end{proof}

\begin{lemma}
\label{componentSwallowing}
Let $g$, $h\in V$ so that they admit no finite non-trivial orbits, and
so that $g$ and $h$ commute.  Suppose $X$ is a component of support
for $g$, and $Y$ is a component of support for $h$.  If $X\cap Y\ne
\emptyset$, then $X\subset\csupp{h}$ and $Y\subset\csupp{g}$.

\end{lemma}
\begin{proof}
Suppose $x\in X\cap Y$.  By Lemma \ref{commonImportantPoints}, we
can choose $x$ so that $x$ is not an important point of $g$ or of $h$.

Let $x_- = \lim_{n\to\infty} xg^{-n}$ and
$x_+ = \lim_{n\to\infty} xg^n$.  It is immediate that these limits
exist and that $x_-,x_+\in$ I($g$).

We further see that both $x_-$ and $x_+$ are important points for $h$
by applying Lemma \ref{importantPointPreservation}.  For example, if
$x_+$ is not in $\csupp{h}$ then there is $n\in N$ with $xg^{n}h =
xg^{n}\neq xhg^{n}$.

Now, as $X$ underlies a connected component $\mathcal{C}$ of the flow
graph of $g$, the set of important points of $g$ in $X$ is actually
contained in the component support of $h$.  Therefore, some
neighborhood of these points in $X$ is actually contained in the
component support of $h$.  It now follows that all of $X$ must be
contained in the component support of $h$, and by symmetry, all of $Y$
must be contained in the component support of $g$.
\end{proof}

The following corollary follows from Lemma \ref{componentSwallowing}.
\begin{corollary}
\label{csuppIntersectionDecomposition}
Let $g$, $h\in V$ so that they admit no finite non-trivial orbits, and
so that $g$ and $h$ commute.  Let $\mathcal{C}_g$ ($\mathcal{C}_h$) be
the components of support of $g$ ($h$) which intersect non-trivially
the components of support of $h$ ($g$) non-trivially.  Under these
circumstances, we have
\[
\csupp{g}\cap \csupp{h} = \cup_{X\in \mathcal{C}_g}X = \cup_{Y\in\mathcal{C}_h}Y.
\]
\end{corollary}

We are finally ready to prove the following lemma, which provides a
fundamental tool in our later analysis.

\begin{lemma}
\label{commonComponent}
Let $g$, $h\in V$ so that they admit no finite non-trivial orbits, and
so that $g$ and $h$ commute.  Suppose $X$ is a component of support
for $g$, and $Y$ is a component of support for $h$.  If $X\cap Y\ne
\emptyset$, then $X = Y$.
\end{lemma} 
\begin{proof}
Suppose $g, h, X, Y$ as in the first two sentences of the statement.

Let $x\in X\cap Y\backslash \Imp{g}$.  Note that such an $x$ exists
and is not an important point for either $g$ or $h$.  Define $x_+ =
\lim_{n\to\infty}xg^n$ as before.

Suppose there is $m\in N$ so that $xh^m\not\in X$.

Since $x_+$ is an important point of $h$ as well as $g$, if $S$ is
large enough, then $x_S = xg^S$ will be close enough to $x_+$ that $m$
applications of $h$ to $x_S$ will result in a point still in a basin
of attraction of $g$ containing $x_+$ (recall such a basin is a Cantor
Set underlying a root of a complimentary component of $R_g\backslash
D_g$ for some representative revealing pair $P_g = (D_g,R_g)$ for
$g$).  In particular, we see that $xh^mg^S$ is not in $X$ while
$xg^Sh^m\in X$.

Our result now follows from the connectivity of the component of the
flow graph of $g$ over $X$, the connectivity of the component of the
flow graph of $h$ over $Y$, the fact that given any important point
$p$ in $X\cap Y$, there is a neighborhood $N_p$ of that point which is
fully contained in both $X$ and $Y$, (so that all orbits in $X$ under
$g$ which enter $N_p$ must be fully contained in $Y$ and all orbits in
$Y$ under $h$ which enter $N_p$ must stay in $X$), and the fact that
every point in $X$ or $Y$ limits to the important points of $g$ and
$h$ under repeated applications of $g$ or $h$.

\end{proof}

The following lemma is reminiscent of a similar result by Brin and
Squier for elements of Thompson's group $F$ (or actually, for
PL$_o$($I$)) in \cite{BSPLR}, and the proof is philosophically the
same (although, the details here are slightly more complicated due to the
presence of extra attractors and repellers).  This lemma combines very
powerfully with the Lemma \ref{commonComponent}.
\begin{lemma}
\label{commonPowersCor}

Suppose that $g$, $h\in V$ so that they admit no finite non-trivial
orbits.  Suppose further that $g$ and $h$ have a common component of
support $X$, and that the actions of $g$ and $h$ commute over $X$.
Then there are non-trivial powers $m$ and $n$ so that $g^m = h^n$ over
$X$.

\end{lemma}
\begin{proof}
Let us suppose $g$ is represented by a revealing pair $P_g =
(D_g,R_g)$, so that given an important point $q$ of $g$, the phrase
``the basin containing $q$'' is well defined (ie., the Cantor Set
underlying the root of the complimentary component of $D_g\backslash
R_g$ or of $R_g\backslash D_g$ which has a repelling or attracting
leaf as a node over $q$).

Fix $p$, an important point of both $g$ and $h$ in $X$.  Since $g$ and
$h$ are affine in a neighborhood $N_p$ of $p$, there are non-trivial
powers $m$ and $n$ so that $g^m=h^n$ on $N_p$.  Now the element
$g^mh^{-n}$ is trivial on $N_p$.  For each $x\in\Imp{g}\cap X$ where
$g^mh^{-n}$ is fixed on some neighborhood of $x$, let $N_x$ be such a
neighborhood, and let $\mathcal{N}$ be the union of these
neighborhoods.

We now have that $\mathcal{N}$ is actually a neighborhood of
$\Imp{g}\cap X = \Imp{h}\cap X$.  Otherwise, there is a pair $r,a\in
\Imp{g}\cap X$ with $a$ an attracting fixed point of $g$ and $r$ a
repelling fixed point of $g$, where one of $\{a,r\}$ is in
$\mathcal{N}$ and the other is not in $\mathcal{N}$, and where there
is a source-sink chain from the basin $B_r$ containing $r$ to the
basin $B_a$ containing $a$ (this follows from the connectivity of the
component of the flow graph of $g$ over $X$).  In the case that $a\in
\mathcal{N}$, there is $x\in \supp{g^mh^{-n}}\cap B_r$ and a positive
power $k$ so that $xg^k\in \mathcal{N}\cap B_a$.  In the case that
$r\in\mathcal{N}$, there is $x\in \supp{g^mh^{-n}}\cap B_a$ so that a
negative power $k$ has $xg^k\in \mathcal{N}\cap B_r$.  In either of
these cases, Lemma \ref{throwSupport} now shows that
$(g^mh^{-n})^{g^k}$ has support where $g^mh^{-n}$ acts as the
identity, which is impossible since $g$ commutes with $g$ and with
$h$.

Now again by Lemma \ref{throwSupport}, $g^mh^{-n}$ cannot have any
support in $X$.  Otherwise, for any $x\in X\cap \supp{g^mh^{-n}}$,
there is a positive power $k$ of $g$ so that $xg^k$ is near to an
attractor of $g$ in $X$, in particular, $k$ can be taken large enough
so that $xg^k\in\mathcal{N}$.  But then, $(g^mh^{-n})^{g^k}$ has
support in $\mathcal{N}$.

\end{proof}

\section{Some free products which do occur in $V$}
In this section we prove Theorem \ref{embedding}.  Our main
constructive tool is the standard Ping-Pong Lemma of Fricke and Klein
\cite{FK}.  

\subsection{Free product recognition}
We give the version of the Ping-Pong Lemma
essentially as it appears in \cite{delaHarpe}.

\begin{lemma}(Ping Pong Lemma)

Let $G$ be a group acting on a set $X$ and let $H_1$, $H_2$ be two subgroups of $G$ such that $|H_1|\geq 3$ and $|H_2|\geq 2$.  Suppose there exist two non-empty subsets $X_1$ and $X_2$ of $X$ such that the following hold:
\begin{itemize}
\item $X_1$ is not contained in $X_2$,
\item for every $h_1\in H_1$, $h_1 \neq 1$ we have $h_1(X_2) \subset X_1$,
\item for every $h_2\in H_2$, $h_2 \neq 1$ we have $h_2(X_1) \subset X_2$,
\end{itemize}
Then the subgroup $H = \langle H_1, H_2\rangle \leq G$ of $G$ generated by $H_1$ and $H_2$ is equal to the free product of $H_1$ and $H_2$: \\
\[
H = H_1 * H_2.
\]
\end{lemma}

In order to make use the Ping-Pong Lemma, we first define a set of
subgroups of $V$ which are easy to use as factors in free product
decompositions.

\subsection{Demonstrative groups}
It is now time to define the demonstrative groups mentioned in the
introduction.  We will say a subgroup $G\leq V$ is
\emph{demonstrative} if and only if there is a node $n\in
\mathcal{T}_2$ so that for every non-trivial $g\in G$, $g$ admits a
revealing pair representation $P_g = (D_g,R_g)$ so that $n$ is a
neutral leaf of $P_g$ and so that $n$ is moved to a different node of
$\mathcal{T}_2$ by the action of $g$.  We call any node
$p\in\mathcal{T}_2$ which satisfies the properties of $n$ in the
definition of a demonstrative group $G$ a \emph{demonstration node for
  $G$}.  As in the introduction, we denote the set of demonstrative
subgroups of $V$ by the symbol $\mathcal{D}$.

Recall that given $n$ a node of $\mathcal{T}_2$, we denote by
$\mathfrak{C}_n$ the Cantor Set underlying $n$.  The following is an
easy dynamical fact pertaining to demonstrative groups.

\begin{lemma}
\label{demonstrationNode}
Let $G\leq V$.  If $G$ is a demonstrative group then there is a
node $n$ of $\mathcal{T}_2$ so that for all $g\in G$,
\[
\mathfrak{C}_ng\cap \mathfrak{C}_n = \emptyset.
\]
\end{lemma}
\begin{proof}
If $G$ is a demonstrative group, then by taking a demonstration node
for $G$ as our $n$ we will produce the desired result, since the
Cantor Set underlying any non-fixed neutral leaf of a tree pair $P$ is
moved entirely off itself by the element of $V$ represented by $P$.
\end{proof}

We now show that the set of demonstrative groups is closed under some
nice operations.
\begin{lemma}
\label{DClosureProps}
Suppose $G$ and $H$ are demonstrative groups, and $G$ is finite.  Then
\begin{enumerate}
\item every subgroup $K\leq H$ is demonstrative, and
\item there is a demonstrative group $L$ with $K\cong G\times H$.
\end{enumerate}
\end{lemma}
\begin{proof}
The first point is immediate from the definition of a demonstrative group.

The second point requires a bit more care.  Let $m$ be a demonstration
node for $G$.  For each element $g\in G$, let $P_g = (D_g,R_g)$ be a
revealing pair for $g$ which has $m$ as a neutral leaf.  Since $G$ is
finite, the orbit of $m$ in $\mathcal{T}_2$ under the action of $G$ is
a finite collection $\mathcal{O}_m$ of nodes of $\mathcal{T}_2$.

We note that Lemma \ref{demonstrationNode} guarantees us that $g_1\neq g_2\in G$
implies $\mathfrak{C}_mg_1\cap\mathfrak{C}_mg_2 = \emptyset$.

Let $n$ be the demonstration node for $H$.  We will find $H_m$, a copy
of $H$, which is demonstrative with demonstration node $mn$ (concatenate
the names of the nodes $m$ and $n$ in $\mathcal{T}_2$), where every
element of $H_m$ commutes with every element of $G$, so that $\langle
G,H_m\rangle\cong G\times H$.

We build $H_m$ as follows.  For every element $h\in H$, let $P_h=
(D_h,R_h)$ be a representative tree pair for $h$ which has $n$ as a
neutral leaf, in accord with the definition of $H$ being demonstrative
with demonstration node $n$.  We now build the tree pair
$P_{\overline{h}} = (D_{\overline{h}},R_{\overline{h}})$ for
$\overline{h}$, the element of $H_m$ which will be the image of $h$
under the embedding sending $H$ to $H_m$.

Let $T$ be a finite binary subtree in $\mathcal{T}_2$ with root the
root of $\mathcal{T}_2$, which contains every node in $\mathcal{O}_m$
(there are infinitely many such if $\mathcal{O}_m$ does not form the
set of leaves for a finite binary subtree of $\mathcal{T}_2$ with root
the root of $\mathcal{T}_2$).  Define $D_{\overline{h}}$ to be the
extension of the tree $T$, where we append the tree $D_h$ to each of
the leaves of $T$ in $\mathcal{O}_m$.  Define $R_{\overline{h}}$ to be
  the extension of $T$ we get when we append the tree $R_h$ to each
  leaf of $T$ in $\mathcal{O}_m$.  Use the identity permutation on the
  leaves of $T$ not in $\mathcal{O}_m$, and for a particular node
  $l\in\mathcal{O}_m$, use the corresponding permutation for $h$ on
  the leaves under $l$ in the domain and range trees
  $D_{\overline{h}}$ and $R_{\overline{h}}$.

It is immediate from construction that this produces a set of elements
$H_m$ so that $\langle H_m\rangle \cong H$.

The node $mn$ is demonstrative for $H_m$.  It is also demonstrative
for $G$.  Let $A$ be a tree which has $n$ as a node.  Now, for each
element $g$ of $G$, simply append $A$ to every leaf in the finite
cycle of neutral leaves containing $m$ for the tree pair $P_g$, to
make a new revealing tree pair for $g$ which now has the node $mn$ as
a neutral leaf (extending the permutation in the obvious fashion).

Since $H_m$ acts the same way under every image of $m$ under the action
of $G$, we see that the elements of $G$ and the elements of $H_m$
commute.

\end{proof}

\begin{remark}
\label{ClosureNodeFacts}

We note that from the proof above that if $G$ is a finite
demonstrative group with demonstration node $m$ and $H$ is a
demonstrative group with demonstration node $n$, then
\begin{enumerate}
\item when passing to a subgroup $K\leq G$, we can still retain $m$ as
  a demonstration node for $K$, and
\item one can use the node $mn$ as the demonstration node of the
  demonstrative direct product representative of $G\times H$.
\end{enumerate}
\end{remark}

\begin{lemma}
\label{DContainmentProps}
There are demonstrative embeddings of 
\begin{enumerate}
\item $Q/Z$ in $V$,
\item every non-trivial cyclic group in $V$, and
\item every finite group in $V$.
\end{enumerate} 
In all of these cases, we can find demonstrative embeddings with
demonstration node ``0''.
\end{lemma}
\begin{proof}
The embedding of $Q/Z$ in $V$ as given in Proposition 5.6 of
\cite{BKMstructure} is demonstrative.  In Figure 2 of that article,
either of the quarters ``01'' and ``11'' can serve as demonstration
nodes for the embedded image of $Q/Z$ in $V$.  More generally, the
node for ``$J_{2,1}$'' given in the embedding described in that
article (which node is chosen by the reader), can serve as the
demonstration node for that embedding.  We note that we can conjugate
any choice of any such embedding of $Q/Z$ in $V$ by an element
$\theta$ of $V$ which sends the node chosen for $J_{2,1}$ to the node
``0''.  Our new embedding will have the node ``0'' as a demonstration
node.

It is quite easy to find demonstrative embeddings of non-trivial
cyclic groups in $V$; the element $g$ given by the revealing pair in
the diagram below generates an infinite cyclic group with
demonstration node ``0''.
\begin{center}
\includegraphics[height=150pt,width=300 pt]{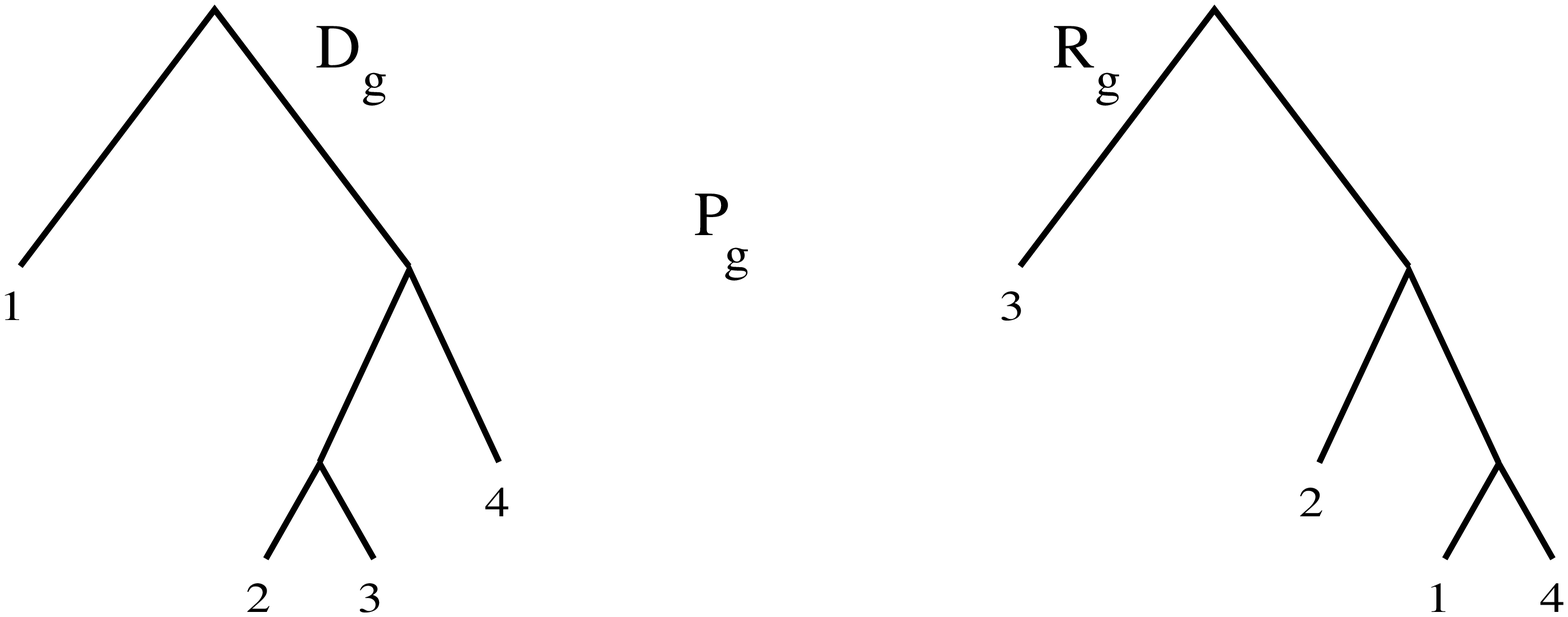}
\end{center}

We now show that for any fixed positive natural number $n$, the
symmetric group on $n$ letters $S_n$ has a demonstrative
representation with demonstration node ``0''.

Let $T$ be a binary tree with $n!$ leaves, which includes ``0'' as a
leaf, with a secondary labeling on leaves using the elements of $S_n$
in some order (with the node ``0'' being labelled with the identity
element).  Represent each element $\alpha$ of $S_n$ by the tree pair
$(T,T)$, and use the permutation which sends the node labelled $m$ to
the node labelled with the group element $m\cdot \alpha$.  It is
immediate that this is a faithful representation of $S_n$, and that
every node is moved by every non-trivial element of $S_n$.  In
particular, the node ``0'' labelled by the identity element in our
initial labelling of $T$ works as well as any other as our
demonstration node.

Passing to a subgroup of a demonstrative group $G$ with demonstration
node $m$ produces a demonstrative group with demonstration node $m$ by
Lemma \ref{DClosureProps} and Remark \ref{ClosureNodeFacts}, so every
finite group has a demonstrative representation with node ``0'' as a
demonstration node.
\end{proof}

In fact, by essentially the argument given above for the case of $Q/Z$,
if we have a demonstrative group $\widehat{G}$, there is a conjugate
version $G$ of $\widehat{G}$ with demonstration node ``0'', we
therefore have the following lemma.
\begin{lemma}
\label{leftRight}
Suppose $G$ is a demonstrative subgroup of $V$.  There is an
isomorphic copy $G_0$ of $G$ with demonstration node ``0'', and
another isomorphic copy $G_1$ of $G$ with demonstration node ``1''.
\end{lemma}

The following is an immediate consequence of Lemma
\ref{DClosureProps}, Lemma \ref{DContainmentProps}, the definition of the
demonstrative groups, and the definition of the class of groups
$\mathcal{A}$.

\begin{corollary}
\label{ADemonstrative}
If $G$ is a group in the class $\mathcal{A}$, then there is a
demonstrative group $K$ so that $G\cong K$.
\end{corollary}

We are now ready to prove our first primary result.

\begin{proof}(of Lemma \ref{DInFFPV} and therefore of Theorem \ref{embedding}):
By Corollary \ref{ADemonstrative}, we need only show that given two
demonstrative groups, we can find a copy of their free product in $V$.

In general, this will follow easily from the Ping-Pong Lemma.  We will
need a separate argument for $Z_2*Z_2$.

Let $G$ and $H$ be non-trivial demonstrative groups, not both
isomorphic with $Z_2$.  Let $X_1 = \mathfrak{C}_1$ and $X_2 =
\mathfrak{C}_0$, that is, the right and left halves of the Cantor Set,
respectively.

Let $G_0$ and $H_1$ be the copies of $G$ and $H$ with demonstration
nodes ``0'' and ``1'', respectively (as in Lemma \ref{leftRight}).

We note that any non-trivial element of $G_0$ takes the whole of the
left half of the Cantor Set $\mathfrak{C}_0$ into the right half
$\mathfrak{C}_1$.  Similarly, any non-trivial element of $H_1$ takes the whole
of the right half of the Cantor Set into the left half of the Cantor
Set.  Therefore, by the Ping-Pong lemma, the group $\langle G_0,
H_1\rangle\cong G_0*H_1\cong G*H$.

Now let $G = \langle g\rangle$ and $H = \langle h\rangle$, where $g$
and $h$ are represented by the tree pairs $P_g = (D_g,R_g)$ and $P_h =
(D_h,R_h)$ below.  Both $g$ and $h$ are order two, so $G\cong Z_2\cong
H$.  However, direct calculation shows that $gh$ has infinite order.
In particular, $\langle G,H\rangle \cong Z_2*Z_2$.
\begin{center}
\includegraphics[height=140pt,width=300 pt]{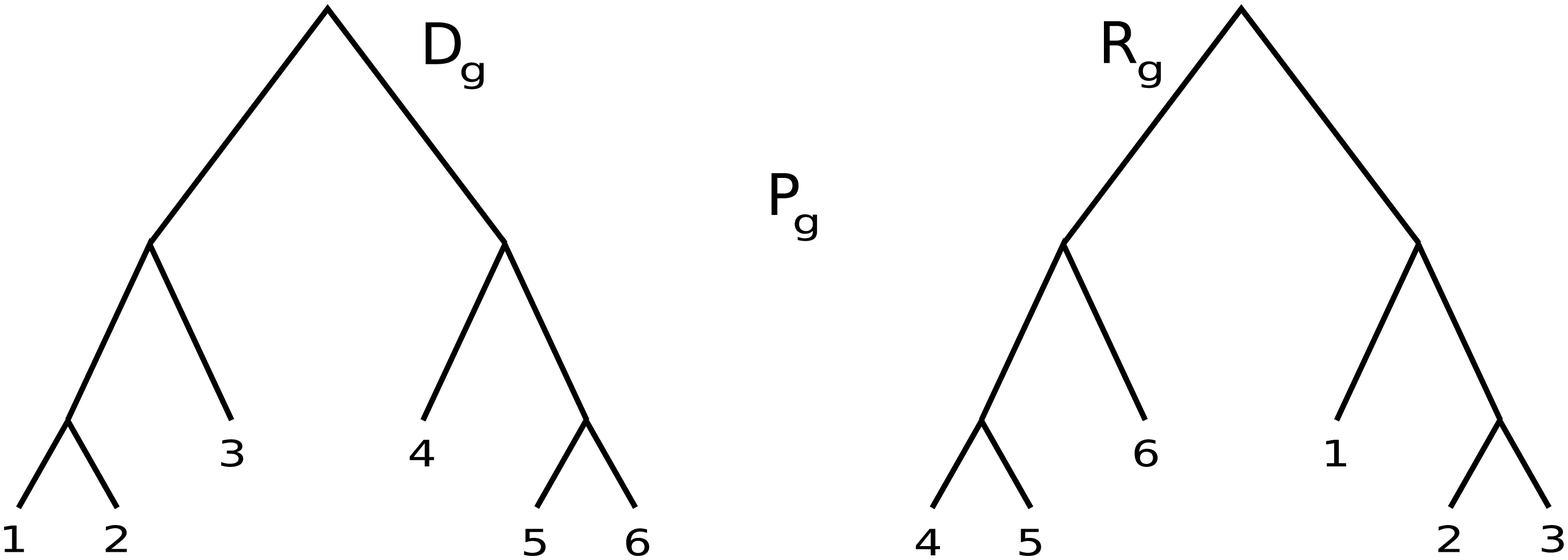}
\includegraphics[height=70pt,width=240 pt]{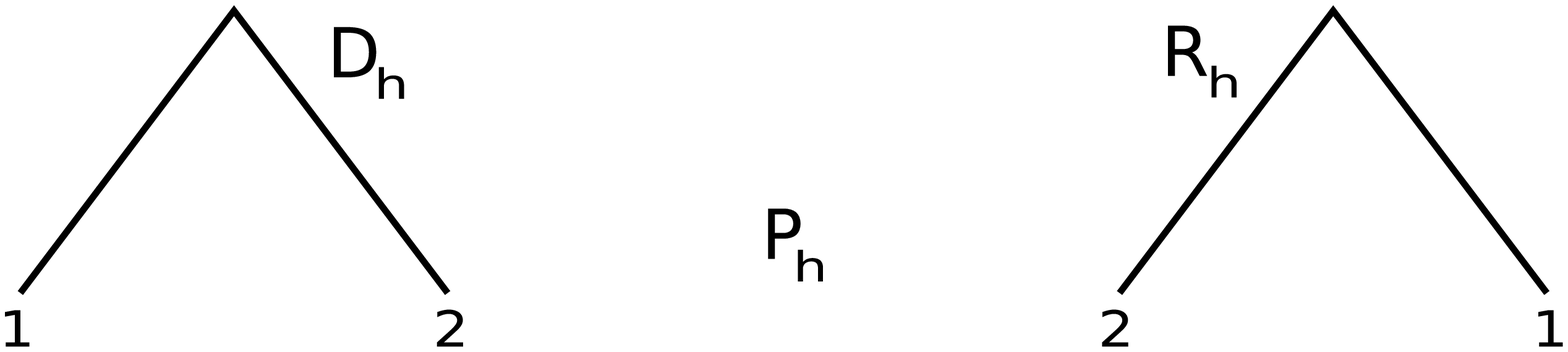}
\end{center}
\end{proof}
\section{Non-embedding results}
We now begin to prove our second primary result, Theorem \ref{nonembedding}.

After we develop some algebraic processes with controlled topological
dynamics, an algorithmic process will demonstrate the non-embedding of
$Z^2*Z$ into $V$.
\subsection{Some commutators in $Z^2 * Z$}
Let $X$ be a non-empty set.  Let $X^{-1}$ be a set disjoint from $X$
in bijective correspondence with $X$.  If $\tau:X \to X^{-1}$ is the
bijection, for each $a\in X$, denote by $a^{-1}$ the element
$\tau{a}$.  If $z \in X^{-1}$, denote by $z^{-1}$ the element
$\tau^{-1}(z)$.  We will call $X$ the alphabet, and for any element
$a\in (X\cup X^{-1})$, we will call $a$ a letter.  We will call any
finite string of letters a word in $X$.  If $w = w_1w_2\cdots w_k$ is
a word in $X$, then we will denote by $w^{-1}$ the word
$w_k^{-1}w_{k-1}^{-1}\cdots w_1^{-1}$.  For any integer $n$, define
the expression $w^n$ as $n$ successive occurrences of the word $w$ if
$n\geq 0$, and $-n$ successive occurrences of the word $w^{-1}$ if $n
< 0$.

Now let $a$, $b$, and $c$ be words in $X$.  We will say that a word
$w$ in $X$ is an \emph{($a,b,c$)-commutator} if there are
minimal integers $n>0$, $x_i$, $y_i$, and $z_i$ with $|x_i| + |y_i| \ne 0$ and
$z_i\ne 0$ for all indices $1\leq i\leq n$, so that
\[
w = [a^{x_1}b^{y_1},[a^{x_2}b^{y_2}, \ldots[a^{x_{n-1}}b^{y_{n-1}},[a^{x_n}b^{y_n}, c^{z_n}]^{z_{n-1}}]^{z_{n-2}}\ldots]^{z_1}].
\] 
Note that in this paper, the commutator bracket $[u,v]$ will always
represent the expression $u^{-1}v^{-1}uv$, as before for general
elements of $V$.

The following is immediate from the definition of an ($a,b,c$)-commutator.
\begin{lemma}
\label{abcInducts}
Suppose $X$ is an alphabet with $a$, $b$, and $c$ words in $X$, and
let $t$ be an ($a,b,c$)-commutator.  If $w$ is an
($a,b,t$)-commutator, then $w$ is an ($a,b,c$)-commutator.
\end{lemma}

In the next lemma, we abuse notation by treating words in the alphabet
$\left\{a,b,c\right\}$ as elements of the group $Z^2*Z$ given by the
presentation $\langle a,b,c\mid [a,b]\rangle$ (recall that by our
definition, words in an alphabet also include ``inverse'' letters).
Given words $w_1$, $w_2$, \ldots,$w_j$, in the alphabet
$\left\{a,b,c\right\}$, we denote by $\langle w_1,w_2,\ldots,
w_j\rangle$ the subgroup of $Z^2*Z$ generated by the elements
represented by these words.  Henceforth, we will freely confuse words
in an alphabet with group elements when it seems unlikely to cause
confusion.

\begin{lemma}

Let $i$, $j$, and $k$ be integers, and define $t=[a^ib^j,c^k]$.  If
$|i|+|j| \neq 0$ and $k\neq 0$, then $\langle a,b,d\rangle$ factors as
$\langle a,b\rangle*\langle d\rangle \cong Z^2 *Z$.

\end{lemma}
\begin{proof}
Let
\[
w = A_0T_0A_1T_1\ldots A_nT_n
\]
where $A_p\in \langle a,b\rangle$ and $T_p\in\langle t\rangle$ for all
valid indices $p$, where $X_p\neq 1$ for $X\in \{A,T\}$ except
possibly for $A_0$ or $T_n$, and where if $n = 0$, one of $A_0$, $T_0$
is non-trivial in $\langle a,b\rangle$ and $\langle t\rangle$
respectively.  We say $w$ is given in form (*).

We will have our lemma if we can show that $w$ does not represent a
trivial element in $Z^2*Z$.

We proceed by induction on $n$.  We will use the phrase
\emph{resultant form} for any expression written as
\[
\prod_{p = 0}^m a^{x_p}b^{y_p}c^{z_p}
\]
where $m$ is an integer with $0\leq m$, $|x_p|+|y_p|\neq 0$ if $p>0$,
and where $z_p\neq 0$ if $p<m$.  We note that in all such forms, the
resulting expression cannot represent the trivial element in $Z^2*Z$
unless $m = x_0=y_0=z_0 = 0$.

We shall prove our result by showing that if $w$ is given as in form
(*), where we further have that $T_n\neq 1$, then $w$ will have
resultant expression ending with one of the two forms (**) below
\[
\begin{array}{c}
c^{-k}a^{fi}b^{fj}c^k\\
c^{-k}a^{-fi}b^{-fj}c^ka^ib^j
\end{array}
\]
(where in both forms $f$ is a positive integer).  In both cases $w$ cannot be
trivial in $Z^2*Z$.

The more general result then follows; elements of the form $a^xb^y$ are not
trivial in $Z^2*Z$ if $|x|+|y|\neq 0$, and any element of $Z^2*Z$ with a
resultant expression ending with either of the forms in (**) cannot be
made trivial by a postmultiplication by a string $a^xb^y$ for integer
values of $x$ and $y$.

We now begin our induction.
\begin{enumerate}
\item Suppose $w = A_0T_0$.

In this base case, if $A_0$ is trivial, the resultant form is $t^z$
for some $z\neq 0$. This expression ends with one of the two forms in
the list (**) (the resultant form depends on the sign of $z$).  In
similar fashion, if both $A_0$ and $T_0$ are not trivial in $\langle
a,b\rangle$ and $\langle t\rangle$ respectively, we then have
\[
w = A_0T_0 = a^xb^yt^z = a^xb^y(a^{-i}b^{-j}c^{-k}a^ib^jc^k)^z
\]

If $z<0$, this word admits no simple cancellations in the group $Z^2*Z$,
and is thus effectively in the resultant non-trivial form in $Z^2*Z$
already (formally, we will need to re-arrange the orders of some $a$'s
and $b$'s after expanding the negative power $z$).  In any case, the
word will end in the form $c^{-k}a^{-i}b^{-j}c^ka^ib^j$, which is in
the bottom form of (**).

If $z>0$, then $w$ simplifies as
\[
w = a^{x-i}b^{y-j}c^{-k}a^ib^jc^k(a^{-i}b^{-j}c^{-k}a^ib^jc^k)^{z-1}
\]
which again is a non-trivial resultant form in $Z^2*Z$, even if the
leading $a^{x-i}$ and $b^{y-j}$ terms are trivial.  In particular,
this word ends in the top form of (**).

\item Suppose now that $n$ is some positive integer and we know by
  induction that for any expression of the form
\[
\prod_{p = 0}^m A_pt^{s_p}
\]
(where $n>m\geq 0$, $A_p\neq 1$ for all $p>0$, and $s_p\neq 0$ for all
$p$) the resultant form of our word ends in one of the two forms in
(**).

We now show that $w$ has resultant form in the list (**).

To begin, we note that $w$ can be expressed as
\[
w=rA_nt^{s_n}, 
\]
where $r$ is expressed in resultant form and ends in one of the two
forms in (**) (by our induction hypothesis), where $s_n\neq 0$, and
where $A_n$ is not trivial in $\langle a,b\rangle$.

There are now two cases in our analysis, each of which splits into two
further subcases.

{\it 2.(a) Suppose $r$ ends with the form $c^{-k}a^{fi}b^{fj}c^k$ where $f>0$, $|i|+|j|\neq 0$, and $k\neq 0$}.

If $s_n<0$ we obtain the following
\[
w = \ldots c^{-k}a^{fi}b^{fj}c^k\cdot A_n\cdot c^{-k}a^{-i}b^{-j}c^ka^ib^j\cdot\ldots\cdot c^{-k}a^{-i}b^{-j}c^ka^ib^j 
\]
 where there are $-s_n$ copies of $c^{-k}a^{-i}b^{-j}c^ka^ib^j $ at
 the end of this expression.  In this case there is aboslutely no
 internal cancelling, and the expression as written above is in
 resultant form (recall, $A_n$ is not trivial in $\langle
 a,b\rangle$).  This expression is in the bottom form of (**).

If $s_n>0$ we obtain
\[
w= \ldots c^{-k}a^{fi}b^{fj}{\bf c^kA_n\cdot
a^{-i}b^{-j}c^{-k}}a^ib^jc^k\cdot\ldots\cdot
a^{-i}b^{-j}c^{-k}a^ib^jc^k
\]
which either fails to cancel the $A_n$ expression with $a^{-i}b^{-j}$, in
which case we obtain the top form of (**), or which has $A_n$ cancel
with $a^{-i}b^{-j}$, in which case the bold substring above cancels
so that we obtain the form
\[
\begin{array}{c}
w= \ldots c^{-k}a^{fi}b^{fj}\cdot a^ib^jc^k\cdot\ldots\cdot
a^{-i}b^{-j}c^{-k}a^ib^jc^k\\
= \ldots c^{-k}a^{(f+1)i}b^{(f+1)j}c^k\cdot
a^{-i}b^{-j}c^{-k}a^ib^jc^k\cdot\ldots\cdot a^{-i}b^{-j}c^{-k}a^ib^jc^k
\end{array}
\] 
in which case we again obtain the top form of (**), even in the case
where $s_n=1$.

{\it 2.(b) Suppose instead that $r$ ends in the form
  $c^{-k}a^{-fi}b^{-fj}c^ka^ib^j$ where $f>0$, $|i|+|j|\neq 0$, and
  $k\neq 0$}. 

 If $s_n<0$ we obtain
\[
w = \ldots c^{-k}a^{-fi}b^{-fj}c^ka^ib^j\cdot A_n\cdot
c^{-k}a^{-i}b^{-j}c^ka^ib^j\cdot\ldots\cdot
c^{-k}a^{-i}b^{-j}c^ka^ib^j.
\]
In the above there are $-s_n$ total occurrences of
$c^{-k}a^{-i}b^{-j}c^ka^ib^j$ at the end of the expression.

If $a^ib^j\cdot A_n$ is trivial in $\langle a,b \rangle$, then the
expression reduces to the resultant form
\[
w = \ldots c^{-k}a^{-(f+1)i}b^{-(f+1)j}c^ka^ib^j\cdot\ldots\cdot
c^{-k}a^{-i}b^{-j}c^ka^ib^j
\]
which is in the lower form from (**), even when $s_n = -1$.  If
$a^ib^j\cdot A_n$ is not trivial in $\langle a,b\rangle$, then there
is less reduction, and again we obtain an expression in the lower form
from (**).

If $s_n>0$ we obtain
\[
w = \ldots c^{-k}a^{-fi}b^{-fj}c^k{\bf a^ib^j\cdot A_n\cdot
  a^{-i}b^{-j}}c^{-k}a^ib^jc^k\cdot\ldots\cdot
a^{-i}b^{-j}c^{-k}a^ib^jc^k.
\]
In this expression there are $s_n$ copies of the subexpression
$a^{-i}b^{-j}c^{-k}a^ib^jc^k$ at the end.  The bold subexpression
$a^ib^j\cdot A_n\cdot a^{-i}b^{-j}$ resolves to $A_n$, which is
non-trivial, and there can be no further reductions, so that we have a
form where no `$c$' sub-expressions can cancel, and we obtain an
expression in the upper form of (**).

Thus, in all cases, our expression for $w$ ends in one of the forms
(**).
\end{enumerate}
\end{proof}

The following follows from the previous two lemmas via a simple induction argument.
\begin{lemma}
\label{abcFree}

Let $Z^2 *Z$ be presented as $\langle a,b,c|[a,b]\rangle$.  If $t$ is
an ($a,b,c$)-commutator, then $\langle a,b,t\rangle$ factors as
$\langle a,b\rangle * \langle t \rangle \cong Z^2 * Z$.
\end{lemma}

\subsection{$Z^2*Z$ cannot embed in $V$}
We now create an algorithm, whose net effect will be to show that
$Z^2*Z$ does not embed into $V$.  We do this by taking a supposed
embedding, and through a process of ``improvements,'' we show that our
embedded group actually admits torsion elements.

Throughout the remainder of the paper, we will assume $G = \langle
a,b,c\mid [a,b]\rangle\cong Z^2 * Z$, and that $\phi:G\to V$ is an
embedding.  Define $\alpha = a\phi$, $\beta = b\phi$, and $\gamma
= c\phi$.

Each subsubsection which follows carries out an algebraic selection of
a subgroup of $\phi(G)$ which must still be isomorphic to $Z^2*Z$ for
various reasons.  The reader should imagine we have improved our
initial selection of $\phi$ at such times, so that we always enter the
next subsubsection with well defined versions of $\alpha$, $\beta$,
and $\gamma$.  In each case, the selection will be motivated by
``improving'' the dynamics of the interactions of $\alpha$, $\beta$,
and $\gamma$.  The reader will need to note how these dynamics are
improving as we progress through the process.
\subsubsection{Improve $\langle\alpha,\beta,\gamma\rangle$ so it acts without non-trivial finite periodic orbits}

Our first step is use Proposition 10.1 of \cite{BrinHigherV} to pass
to powers of $\alpha$, $\beta$, and $\gamma$ so that none of the
resulting elements admit non-trivial finite orbits.  In this
subsubsection we will also develop some notation for various sets and
quantities which will be important in the remainder.

By Lemma \ref{commonComponent} we see that $\alpha$ and $\beta$
have various components of support, some of which may be common to both
elements.  Each of the remaining components of support of
$\alpha$ or $\beta$ will be disjoint from the support of $\beta$ or
$\alpha$, respectively.  In particular, we can give name to these sets
as follows.  Let $\left\{A_i\right\}_{i = 1}^r$ represent the
components of support of $\alpha$ which are disjoint from the
components of support of $\beta$.  Let $\left\{B_j\right\}_{j = 1}^s$
represent the components of support of $\beta$ which are disjoint from
the components of support of $\alpha$.  Finally, let
$\left\{C_k\right\}_{k = 1}^t$ represent the common components of
support of $\alpha$ and $\beta$.

For each component of support $C_k$, fix non-trivial integers $m_k$
and $n_k$ so that $\alpha^{m_k}\beta^{n_k}$ has trivial action over
$C_k$.  These integers exist by Corollary \ref{commonPowersCor}.

\subsubsection{Modify $\gamma$ so that $I(\gamma)\cap(I(\alpha)\cup I(\beta))=\emptyset$}
In this subsubsection we improve $\phi$ in stages.  We repeatedly
replace $\gamma$ with ($\alpha,\beta,\gamma$)-commutators.  Our goal
is to arrange the components of support of the final version of
$\gamma$ so that the set of important points of $\gamma$ is disjoint
from the sets of important points of both $\alpha$ and $\beta$.  The
resulting group $\langle \alpha, \beta, \gamma\rangle$ must still
factor as $\langle \alpha,\beta\rangle * \langle \gamma\rangle \cong
Z^2*Z$ by Lemma \ref{abcFree}.

Let $S = I(\gamma)\cap(I(\alpha)\cup I(\beta))$.  If $S$ is non-empty,
then define the ($\alpha,\beta,\gamma$)-commutator 
\[
\theta \mathrel{\mathop:}= [\alpha\beta,\gamma].
\]
Otherwise, we already have the goal of the subsubsection and we can
pass to the next subsubsection without modifying $\gamma$.

If $x\in S$, then either $x$ is an important point of $\alpha\beta$,
or $\alpha\beta$ acts as the identity in a neighborhood of $x$
($\alpha$ and $\beta$ may act as local inverses in a neighborhood of
$x$).  In either case, $\theta$ will act as the identity in a
neighborhood of $x$ (either by invoking Lemma
\ref{commonImportantPoints}, or simply, if $\alpha\beta$ acts
trivially near $x$, then the commutator resolves as
$\gamma^{-1}\gamma$ in a neighborhood of $x$).

Improve $\gamma$ by replacing it with $\theta$.

We now should mention a useful lemma.
\begin{lemma}
\label{noRevisit}

If $y$ is an important point of $\alpha$ or $\beta$, and $\gamma$ acts
as the identity in some neighborhood $\mathcal{M}_y$ of $y$, then
any ($\alpha,\beta,\gamma$)-commutator $\tau$ will act as the identity
in some neighborhood $\mathcal{N}_y$ of $y$.
\end{lemma}
\begin{proof}
We show that if $p$, $q$ are integers with $|p|+|q|>0$, $z\neq
0$ is an integer, and $\tau = [\alpha^p\beta^q,\gamma^z]$, then
$y$ has a neighborhood $\mathcal{N}_y$ so that $\tau$ acts as the
identity on $\mathcal{N}_y$.  The general lemma then follows by an
easy induction.

Let $p$, $q$, $z$ and $\tau$ as in the previous paragraph.  We have
\[
\tau = (\gamma^{-1})^{\alpha^p\beta^q}\cdot \gamma.
\] 
In particular, the support of $\tau$ is contained in the union
Supp($\gamma$)$\cup$ Supp($\gamma$)$\alpha^p\beta^q$.  Let
$\mathcal{M}_y$ be a neighborhood of $y$ disjoint from the action of
$\gamma$, and let $m$ be the node of $\mathcal{T}_2$ corresponding to
$\mathcal{M}_y$.  Pass deeply into $\mathcal{T}_2$ beneath the node
$m$ to find a node $n$ which has $y$ underlying $n$, and so that
$\alpha^p\beta^q$ acts affinely over the Cantor Set $\mathcal{N}_y$
underlying $n$, and so that $\mathcal{N}_y\alpha^{-p}\beta^{-q}\subset
\mathcal{M}_y$.  It is immediate that such a node $n$ exists. As
Supp($\gamma$) lies outside of $\mathcal{M}_y$, the action of
$\alpha^p\beta^q$ cannot throw the support of $\gamma$ into
$\mathcal{N}_y$.
\end{proof}

If our new $\gamma$ has new important points in common with the
important points of $\alpha$ or the important points of $\beta$, then
return to the beginning of this subsubsection, observing that by the lemma
we have just proven, any further versions of $\gamma$ that we create
will act as the identity in a neighborhood of the set $S$ which we
defined in the beginning of this section.

This process must stop, since the important points of $\alpha$ and
$\beta$ are finite in number.  The final version $\gamma$ that we have
created has no important points in common with the important points of
$\alpha$ or the important points of $\beta$.

Proceed to the next subsubsection.

\subsubsection{Improve $\gamma$ so that ${\rm Supp}(\gamma)\,\cap\,\left(I(\alpha)\cup I(\beta)\right) = \emptyset$}
This process requires a bit more care.

We may suppose immediately that $I(\gamma)\cap (I(\alpha)\cup
I(\beta)) = \emptyset,$ or we could not have entered this subsubsection
in our algorithm.

If Supp($\gamma$) does not contain any of the important points of
$\alpha$ or $\beta$, then proceed to the next subsubsection, otherwise
continue in this subsection.

Suppose $x$ is an important point of $\alpha$ or $\beta$ which is in
the support of $\gamma$.  We must carefully consider two cases,
depending on the dynamics of $\alpha$ near $y = x\gamma^{-1}$.

In the first case, $y$ is disjoint from the support of $\alpha$.

  In this case, define $\theta = [\alpha,\gamma] =
  \alpha^{-1}\gamma^{-1}\alpha\gamma$.  We observe the following.

\[
x\theta = x\alpha\cdot \gamma^{-1}\alpha\gamma =
x\gamma^{-1}\alpha\gamma = y\alpha\gamma = y\gamma = x
\]

So $x$ is fixed by $\theta$.  

If $x$ is an important point of $\alpha$, then unless $y$ is also an
important point of $\alpha$, $x$ will be an important point of
$\theta$ as well; the initial invocation of $\alpha^{-1}$ acts as an
affine map fixing $x$ with slope not equal to one in a neighborhood of
$x$, while the latter invocation of $\alpha$ acts as the identity near
$y$.  If $y$ is an important point of $\alpha$ as well, then $x$ will
be an important point for $\theta$ only if the slope of $\alpha$ in
small neighborhoods of $x$ is not the same as the slope of $\alpha$ in
small neighbrohoods of $y$ (otherwise, $\theta$ will act as the
identity in some neighborhood of $x$).

If $x$ is not an important point of $\alpha$, then $x$ is an important
point of $\beta$ which must be disjoint from the component support of
$\alpha$.  In this case, $\alpha$ acts as the identity in a
neighborhood of $x$. Therefore, $\theta$ will either have $x$ as an
important point (if $y$ is an important point of $\alpha$) or
act as the identity in a neighborhood of $x$.

Suppose instead that $y$ is not disjoint from the support of $\alpha$.
Then either $y$ is in a common component of support $C_k$ of $\alpha$
and $\beta$, or $y$ is in a component of support $A_i$ for $\alpha$
for some index $i$.  In the first case we can use Corollary
\ref{commonPowersCor} to find $p$, $q$ non-trivial integers so that
$\alpha^p\beta^q$ acts trivially over $C_k$, while in the second case,
take $p = 0$ and $q = 1$, so that again $\alpha^p\beta^q$ is trivial
over $A_i$.  In either case, define $\theta =
[\alpha^p\beta^q,\gamma]$, and $\theta$ will have $x$ as an important
point if $x$ is an important point of $\alpha^p\beta^q$, as in the
previous discussion, or $\theta$ will act trivially in a neighborhood
of $x$.

Now replace $\gamma$ by $\theta$.  If $\gamma$ has any important
points in common with the important points of $\alpha$ or $\beta$,
return to the previous subsubsection.  Otherwise, if there are any
further important points of $\alpha$ or $\beta$ in the support of
$\gamma$, then return to the beginning of this subsubsection.
Finally, if none of the important points of $\alpha$ or $\beta$ are in
the component support of $\gamma$, we can proceed to the next
subsubsection.

Note that if we had ${\rm Supp}(\gamma)\,\cap\,\left(I(\alpha)\cup
I(\beta)\right) \ne \emptyset$ at the beginning of this subsubsection,
then the cardinality of that intersection is reduced by at least one
by the process here (no new points of $I(\alpha)\cup I(\beta)$ enter
the support of $\gamma$ under this process by Lemma \ref{noRevisit}).
Further, again by Lemma \ref{noRevisit}, this subsubsection will only
force the algorithm to return to the previous subsubsection a finite
number of times, and the cardinality of ${\rm
  Supp}(\gamma)\,\cap\,\left(I(\alpha)\cup I(\beta)\right)$ does not
increase when we apply the process of the previous section.
Therefore, the algorithm given to this point eventually passes to the
next subsubsection, and at that time none of the important points of
$\alpha$ or of $\beta$ are in the component support of $\gamma$.

\subsubsection{Finding torsion where none can exist}
We can now assume that there is a neighborhood of the important points
of $\alpha$ and $\beta$ so that $\gamma$ acts trivially over this
neighborhood.

In this section, we find an element in
$\langle\alpha,\beta,\gamma\rangle$, which is the $\phi$-image of a
non-trivial element of $Z^2*Z$, but where orbit dynamics show that the
image element must actually either be trivial, or torsion with order
two, three, or six.  Thus, $\phi$ cannot actually embed $Z^2*Z$ into
$V$.

We first note the effects of conjugating $\gamma$ by powers of
$\alpha$ and $\beta$, assuming the details of our current dynamical
situation with respect to important points and supports.  

\begin{lemma} 
\label{gammaPastItself}

Suppose $\delta$, $\epsilon$, and $\theta$ are elements in $V$ with
the properties that
\begin{enumerate}
\item $\delta$ and $\epsilon$ commute, 
\item none of $\delta$, $\epsilon$ or $\theta$ admit non-trivial
  finite orbits in their action on the Cantor Set,
\item the support of $\theta$ is disjoint from a neighborhood of the
  important points of $\delta$ and $\epsilon$.
\end{enumerate}
There are infinitely many pairs of non-zero integers $x$ and $y$ so that \\

Supp($\theta^{\delta^x\epsilon^y}$) $\cap$ Supp($\theta$) $\cap$
(Supp($\delta$) $\cup$ Supp($\epsilon$)) $= \emptyset$.
\end{lemma}
\begin{proof}
The lemma follows immediately from the observation that there are
non-zero integers $i$ and $j$ so that $\delta^i\epsilon^j$ is
non-trivial over every component of common support of $\delta$ and
$\epsilon$ (since there are only finitely many such components), and
thus over every component of support of $\delta$ and $\epsilon$.  Now
for large enough integers $n$, setting $x =ni$ and $y = nj$ produces
integers so that $\delta^x\epsilon^y$ throws the support of $\theta$
entirely off of itself within the supports of $\delta$ and $\epsilon$
(the resulting support within the support of $\delta$ and $\epsilon$
will be near to the important points of $\delta$ and $\epsilon$).
\end{proof}
Thus we may use conjugation of $\gamma$ by powers of $\alpha$ and
$\beta$ to find non-trivial elements in
$\langle\alpha,\beta,\gamma\rangle$ whose actions \emph{within the
  supports of $\alpha$ and $\beta$} are disjoint from the action of
$\gamma$.

Let $x$ and $y$ be the integers guaranteed by Lemma
\ref{gammaPastItself}, where $\alpha$ plays the role of $\delta$,
$\beta$ plays the role of $\beta$, and $\gamma$ plays the role of
$\theta$.  Define $\omega = [\gamma,\gamma^{\alpha^x\beta^y}]$, noting
in passing that $[c,c^{a^xb^y}]$ is not trivial in $Z^2*Z$.

Note that the component support of $\omega$ is contained in the set
$\Gamma=\overline{{\rm Supp(}\gamma{\rm )}}\cup \overline{{\rm
    Supp(}\gamma^{\alpha^x\beta^y}{\rm )}}$.  One can now demonstrate
by direct calculation that every $x\in\Gamma$ must travel, under the
action of $\langle \omega\rangle$, along an orbit of length either
one, two, or three.  This implies that the order of $\omega$ is one,
two, three, or six, which contradicts the fact that $\omega$ is a
non-trivial element in a torsion free group.  In particular, there can
be no embedding of $Z^2*Z$ in $V$.

The remainder of this subsubsection verifies the calculation mentioned
in the previous paragraph.  The discussion is highly technical. We
first carefully define fifteen disjoint subsets of the potential
support of $\omega$.  Then, we analyze the flow of points in these
sets under the action of $\langle \omega\rangle$.

To assist the reader in tracking this process, the following schematic
provides an informal guide to the arrangements of the fifteen sets in the
context of the actions of $\gamma$, $\theta$, and $\alpha^x\beta^y$.

\begin{center}
\psfrag{Uc}{$U_c$}
\psfrag{Vc}{$V_c$}
\psfrag{Ug}{$U_g$}
\psfrag{Vg}{$V_g$}
\psfrag{Ut}{$U_t$}
\psfrag{Vt}{$V_t$}
\psfrag{Mc}{$M_c$}
\psfrag{Mg}{$M_g$}
\psfrag{Mt}{$M_t$}
\psfrag{axby}{$\alpha^x\beta^y$}
\psfrag{1}{$1$}
\psfrag{2}{$2$}
\includegraphics[height=180pt,width=250 pt]{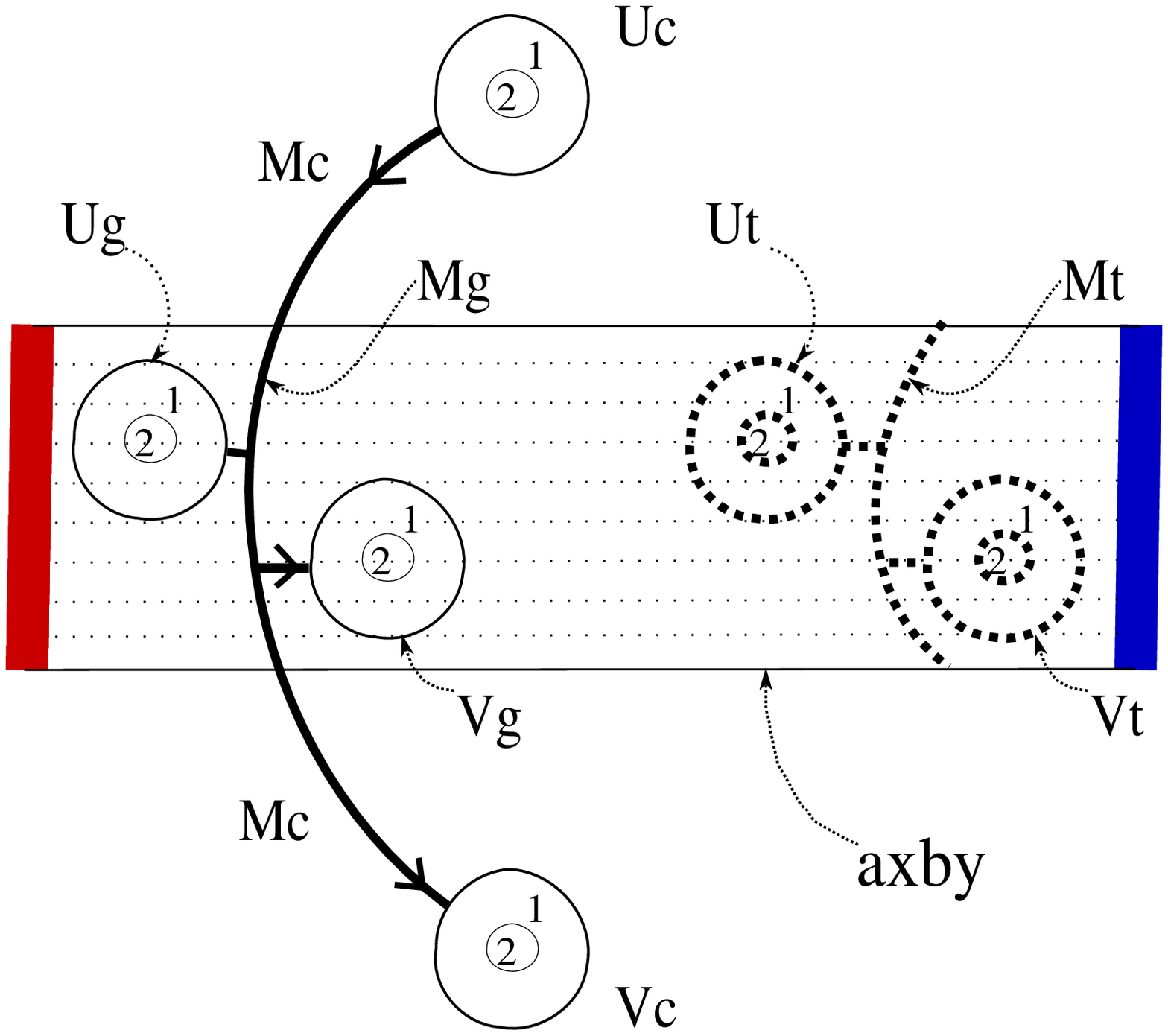}
\end{center}

\newpage
\underline{Defining sets and notation}:\\

To simplfy notation, set $\theta = \gamma^{\alpha^x\beta^y}$.  Set
\[
\mathcal{I}={\rm I(}\gamma{\rm )}\cup {\rm I(}\theta{\rm )}.
\]
While the important points of $\gamma$ and $\theta$ cannot be
important points of $\omega$, the dynamical analysis of the action of
$\omega$ is greatly assisted by paying careful attention to a
neighborhood of the points in $\mathcal{I}$.

\newcommand{\axby}{\alpha^x\beta^y}
\newcommand{\suppab}{{\rm Supp(}\axby{\rm )}}
\newcommand{\suppg}{{\rm Supp(}\gamma{\rm )}}
\newcommand{\suppt}{{\rm Supp(}\theta{\rm )}}

The set $\mathcal{I}$ decomposes as a disjoint union of six sets, some of which might be empty.
\[
\mathcal{I} = R_c\cup R_g\cup R_t\cup A_c\cup A_g\cup A_t
\]
Here, $R_c$ is the set of repelling fixed points of $\gamma$ and
$\theta$ which lie outside of the support of $\alpha^x\beta^y$.  We
note in passing that this is the same set for both $\gamma$ and
$\theta$ (subscript $c$ denotes the word ``common'').

Similarly, $A_c$ is the set of attracting fixed points of $\gamma$ and $\theta$ which lie outside the support of $\alpha^x\beta^y$.

The sets $R_g$ and $A_g$ are respectively the sets of repelling and attracting
fixed points of $\gamma$ which happen to lie in the support of
$\alpha^x\beta^y$.

The sets $R_t$ and $A_t$ are respectively the sets of repelling and
attracting fixed points of $\theta$ which lie in the support of
$\alpha^x\beta^y$.  One sees immediately that $R_t =
R_g\alpha^x\beta^y$ and $A_t = R_t\alpha^x\beta^y$.

For each $x\in R_c$, let $n_x$ denote a node of $\mathcal{T}_2$ so
that $x\in\mathcal{C}_{n_x}\subset\overline{\supp{\gamma}}$ and so
that $\mathcal{C}_{n_x}\cap {\rm Supp(}\alpha^x\beta^y{\rm )} =
\emptyset$.  Similarly, for each $y\in A_c$, let $\hat{n}_y$
denote a node of $\mathcal{T}_2$ so that
$y\in\mathcal{C}_{n_y}\subset\overline{\supp{\gamma}}$ and so that
$\mathcal{C}_{n_y}\cap{\rm Supp(}\alpha^x\beta^y{\rm )} = \emptyset$.

For each $x\in R_g$, let $n_x$ denote a node of $\mathcal{T}_2$ so
that $x\in\mathcal{C}_{n_x}\subset\overline{\supp{\gamma}}$ and so
that $\mathcal{C}_{n_x}\cap {\rm Supp(}\alpha^x\beta^y{\rm )} =
\mathcal{C}_{n_x}$.  Similarly, for each $y\in A_g$, let
$\hat{n}_y$ denote a node of $\mathcal{T}_2$ so that
$y\in\mathcal{C}_{\hat{n}_y}\subset\overline{\supp{\gamma}}$ and so
that $\mathcal{C}_{\hat{n}_y}\cap{\rm Supp(}\alpha^x\beta^y{\rm )} =
\mathcal{C}_{\hat{n}_y}$.

For each $x\in R_t$, let $n_x$ denote a node of $\mathcal{T}_2$ so
that $x\in\mathcal{C}_{n_x}\subset\overline{\supp{\theta}}$ and so
that $\mathcal{C}_{n_x}\cap {\rm Supp(}\alpha^x\beta^y{\rm )} =
\mathcal{C}_{n_x}$.  Similarly, for each $y\in A_t$, let
$\hat{n}_y$ denote a node of $\mathcal{T}_2$ so that
$y\in\mathcal{C}_{\hat{n}_y}\subset\overline{\supp{\theta}}$ and so
that $\mathcal{C}_{\hat{n}_y}\cap{\rm Supp(}\alpha^x\beta^y{\rm )} =
\mathcal{C}_{\hat{n}_y}$.

All of the nodes $n_x$ and $\hat{n}_y$ chosen above can be chosen in
such a fashion as to have all of the various properties mentioned
above, as well as the further property that given any distinct pair of
nodes, the underlying Cantor Sets of the nodes are disjoint.  We
assume that the nodes have been chosen in such a fashion.  The reader
may verify that such choices can be made.

Let $R = R_c\cup R_g\cup R_t$, and let $A = A_c\cup A_g\cup A_t$, and set 
\[
\begin{array}{c}
\mathcal{N}_r = \cup_{x\in R}\mathfrak{C}_{n_x}\\
\widehat{\mathcal{N}}_a = \cup_{y\in A}\mathfrak{C}_{\hat{n}_y}.
\end{array}
\]
The set $\mathcal{N}_r$ is a neighborhood of the repelling fixed
points of $\gamma$ and $\theta$, which we will use often in our
calculations below.  A modified version of $\widehat{\mathcal{N}}_a$
will play the role of a similar neighborhood around the
attractors.

There is an integer $K>0$ so that $\mathcal{N}_r\gamma^K\cup
\mathcal{N}_r\theta^K\cup \widehat{\mathcal{N}}_a = \suppg\cup\suppt$.
Replace $\gamma$ by $\gamma^K$, and redefine $\theta$ as
$\gamma^{\axby}$ using the new $\gamma$.  We now have
\[
\mathcal{N}_r\gamma\cup \mathcal{N}_r\theta\cup \widehat{\mathcal{N}}_a = \suppg\cup\suppt.
\]

For each $x\in R$, set $U_x = \mathfrak{C}_{n_x}$.  For each $y\in A$,
set $V_y = \mathfrak{C}_{\hat{n}_y}\backslash(\mathcal{N}_r\gamma\cup
\mathcal{N}_r\theta)$.  Given $x\in R_c$, $R_g$, or $R_t$
respectively, we may denote $U_x$ by $U_x^c$, $U_x^g$, or $U_x^t$ for
clarity.  Similarly, given $y\in A_c$, $A_g$, or $A_t$, we may denote
$V_y$ by $V_y^c$, $V_y^g$, or $V_y^t$ respectively.

Set $\mathcal{N}_a= \cup_{y\in A}V_y$.  This is a generalized
neighborhood of the attracting fixed points of $\gamma$ and $\theta$
lying entirely in the supports of $\gamma$ and $\theta$.

Denote by $M$ the set
$(\suppg\cup\suppt)\backslash(\mathcal{N}_r\cup\mathcal{N}_a)$.  This is
the set of potential support of $\omega$ away from the neighborhoods
of the attracting and repelling fixed points of $\gamma$ and $\theta$.
(The set $M$ is the ``middle'' of the potential support of $\omega$.)
The set $M$ decomposes as a disjoint union of sets $M_c =
M\backslash\suppab$, $M_g = M\cap \suppg\cap\suppab$, and $M_t =
M\cap\suppt\cap\suppab$.

We now use $M$ to further decompose the sets $U_x$ and $V_y$ for $x\in
R$ and $y\in A$.  

Given $x\in U_x$ set
\[
U_{x,1} = \left\{x\in U_x\mid x\gamma\in M \textrm{ or } x\theta\in M\right\}
\]
and set $U_{x,2} = U_x\backslash U_{x,1}$. (At least two applications
of $\gamma$ or $\theta$ are required for a point in $U_{x,2}$ to leave
$U_x$.)  Extend this notation in the obvious fashion so that
$U_{x,i}^c$, $U_{x,i}^g$, and $U_{x,i}^t$ are well defined for $i = 1$
or $i = 2$.  We will not need larger values of $i$ ($i$ loosely
represents minimal ``escape time'' from $U_x$ under the actions of
$\gamma$ or $\theta$) for the analysis to follow.

Finally, and in similar fashion, given $y\in R$, define
\[
V_{y,1}= V_y\cap (M\theta\cup M\gamma)
\]
and set $V_{y,2} = V_y\backslash V_{y,1}$.  Extend this notation in
like fashion so that the sets $V_{y,i}^c$, $V_{y,i}^g$, and
$V_{y,i}^t$ are well defined for indices $i = 1$ and $i = 2$.

We have now defined the fifteen types of (pairwise disjoint) subsets
of the potential support of $\omega$ which we will require in our
analysis of the action of $\omega$ on $\mathfrak{C}$.  In particular,
these are $U_{x,i}^c$, $U_{x,i}^g$, $U_{x,i}^t$, $V_{x,i}^c$,
$V_{x,i}^g$, $V_{x,i}^t$, $M_c$, $M_g$, and $M_t$, where $i =1$ or $i
= 2$.

\underline{Analysis of dynamics}:\\
We are now ready to prove the following lemma.
\begin{lemma}
The element $\omega$ defined above has order $1$, $2$, $3$ or $6$.
\end{lemma}
\begin{proof}
The support of $\omega$ is contained in the union of the support of
$\gamma$ and the support of $\theta$.  In particular, we need to trace
the orbits of the points in each of the fifteen sets defined in the
paragraph just before the ``Analysis of dynamics'' header.

That effort is simplified by the fact that the action of
$\langle\omega\rangle$ in many of those sets is trivial (as one might
expect, given that $\omega$ is a commutator).

We perform the orbit calculations for some of the sets before confirming
the statement in the last paragraph, in order to acquaint the reader
with a method of orbit calculation.

In the diagram below, if there are several arrows leaving a node, this
represents the fact that a point may move to distinct locations
depending on further subdivisions within the fifteen sets.  It often
occurs that a previous choice at a branch makes later choices invalid.
In our first calculation, we will draw such invalid possibilities with
a ``dotted'' arrow.  Arrows are sometimes decorated with strings to
help explain the dynamics.

If $x\in R$, in the diagrams below, we will drop the occurrence of $x$
from the names of the repulsive sets $U_{x,1}^c$, $U_{x,2}^c$,
$U_{x,1}^g$, $U_{x,2}^g$, $U_{x,1}^t$, and $U_{x,1}^t$, writing
instead names such as $U_1^c$.  We will treat the attractive $V$ sets
in similar fashion.  This should lead to no confusion.

The reader may be assisted in following the calculations below by
recalling that $M_t = M_g\axby$, $U_*^t = U_*^g\axby$, $V_*^t =
V_*^g\axby$, and that these sets are ``parallel'' in some sense due to
the relationship $\mathcal{O}(p,\theta) = \mathcal{O}(p,\gamma)\axby$
for $p\in \mathfrak{C}$.  The reader should also observe that
$\gamma|_{V_*^c}=\theta|_{V_*^c}$ and
$\gamma^{-1}|_{U_*^c}=\theta^{-1}|_{U_*^c}$ for $* = 1$ or $* = 2$.

We now begin to trace orbits.

Assume $x_0\in M_c$.
\[
\xymatrix{
{x_0\in M_c} \ar[r]^{\gamma^{-1}}
             \ar[ddr]^{\gamma^{-1}} 
                     & {x_1\in U_1^c} \ar[r]^{\theta^{-1}}     & {x_2\in U_2^c} \ar[r]^{\gamma} & {x_1} \ar[r]^{\theta}
                                                                                                     \ar@{.>}[dr]^{\theta} & {x_0}\\
& & & &{M_t} \\
  & {x_3\in U_1^g} \ar[r]^{\theta^{-1}}_= & {x_3} \ar[r]^{\gamma} & {x_0} \ar[r]^{\theta} 
                                                                    \ar[dr]^{\theta}_{=\gamma} & {x_4\in V_1^t}\\
& & & & {x_5\in V_1^c}\\
x_4\ar[r]^{\gamma^{-1}}_{=}&x_4\ar[r]^{\theta^{-1}}&x_0\ar[r]^{\gamma}\ar@{.>}[dr]^{\gamma}&x_6\in V_1^g\ar[r]^{\theta}_{=}&x_6\\
 & & & {V_1^c}\\
x_5\ar[r]^{\gamma^{-1}}\ar@{.>}[dr]^{\gamma^{-1}}&x_0\ar[r]^{\theta^{-1}}&x_7\in U_1^t\ar[r]^{\gamma}_{=}&x_7\ar[r]^{\theta}&x_0\\
 & M_g\\
x_6\ar[r]^{\gamma^{-1}}&x_0\ar[r]^{\theta^{-1}}&x_8\in U_1^t\ar[r]^{\gamma}_{=}&x_8\ar[r]^{\theta}&x_0
}
\]

In this example, each diagram component represents the possibilities
of one application of $\omega$.  Under the action of $\langle \omega
\rangle$, the point $x_0$ had potential orbits of length one
($\{x_0\}$), of length two ($\{x_0, x_5\}$), and of length three
($\{x_0,x_4,x_6\}$).

We now free the symbols $x_i$ used in the above calculation, so that
we can enter into similar calculations below with different values for
the $x_i$.  Whenever we are about to trace the orbits of one of the
fifteen sets, we will assume the variables $x_i$ are unbound and
available.

In each of the calculations below, we will no longer draw explicit
``dotted arrows'' for potential branches which cannot actually occur
given previous information within the calculation.

Suppose $x_0\in M_t$.
\[
\xymatrix{
{x_0\in M_t} \ar[r]^{\gamma^{-1}}_=& {x_0} \ar[r]^{\theta^{-1}} \ar[dr]^{\theta^{-1}} 
             & {x_1\in U_1^t} \ar[r]^{\gamma}_= & {x_1} \ar[r]^{\theta}&x_0\\
 & & {x_2\in U_1^c}\ar[r]^{\gamma}&{x_3\in M_g}\ar[r]^{\theta}_=&{x_3}\\
{x_3}\ar[r]^{\gamma^{-1}}& {x_2}\ar[r]^{\theta^{-1}}&{x_3\in U_2^c}\ar[r]^{\gamma}&{x_2}\ar[r]^{\theta}&{x_0}
}
\]
We thus see that the orbits of points in $M_t$ are either trivial or of length two under the action of $\langle \omega\rangle$.

Suppose now that $x_0\in M_g$.
\[
\xymatrix{
{x_0\in M_g} \ar[r]^{\gamma^{-1}}\ar[dr]^{\gamma^{-1}}& {x_1\in U_1^g} \ar[r]^{\theta^{-1}}_= 
             & {x_1}\ar[r]^{\gamma}& {x_0}\ar[r]^{\theta}_=&{x_0}\\
 & {x_2\in U_1^c}\ar[r]^{\theta^{-1}}&{x_3\in U_2^c}\ar[r]^{\gamma}&{x_2}\ar[r]^{\theta}& {x_4\in M_t}\\
{x_4}\ar[r]^{\gamma^{-1}}_=& {x_4}\ar[r]^{\theta^{-1}}&{x_2}\ar[r]^{\gamma}&{x_0}\ar[r]^{\theta}_=&{x_0}
}
\]
So the points in $M_g$ also have only trivial orbits or orbits of
length two under the action of $\langle \omega\rangle$.

Armed with these previous examples the reader should now be able to
check that if an initial point $p$ has $p\in U_\#^*$, then for any
valid values of the symbols $\#$ and $*$, $\omega$ fixes $p$.  For
$p\in V_2^*$ with $*$ not ``c'', it is also easy to see that $p\omega
=p$.  Thus, we have seen that our lemma is supported over the sets
$M_c$, $M_g$, $M_t$, $U_1^c$, $U_2^c$, $U_1^g$, $U_2^g$, $U_1^t$,
$U_2^t$, $V_2^g$ and $V_2^t$.  There remain four sets to check.

Suppose now that $x_0\in V_2^c$.
\[
\xymatrix{
{x_0\in V_2^c}  \ar[r]^{\gamma^{-1}}
                \ar[ddr]^{\gamma{-1}} & {x_1\in V_1^c} \ar[r]^{\theta^{-1}} 
                                                      \ar[dr]^{\theta^{-1}} & {x_2\in M_c}\ar[r]^{\gamma} & {x_1}\ar[r]^{\theta}&{x_0}\\
& &{x_3\in M_t} \ar[r]^{\gamma}_=&{x_3}\ar[r]^{\theta}&{x_1}\\
&{x_4\in V_2^c} \ar[r]^{\theta^{-1}}
                \ar[dr]^{\theta^{-1}} & {x_5\in V_1^c} \ar[r]^{\gamma} & {x_4}\ar[r]^{\theta} &{x_0}\\
& & {x_6\in V_2^c} \ar[r]^{\gamma} & {x_4}\ar[r]^{\theta} & {x_0}\\
{x_1} \ar[r]^{\gamma^{-1}}&{x_7\in M_g} \ar[r]^{\theta^{-1}}_=&{x_7}\ar[r]^{\gamma}&{x_1}\ar[r]^{\theta} & {x_0}\\
}
\]
In particular, the orbits in $\mathfrak{C}$ under the action of
$\langle\omega\rangle$ which actually intersect $V_2^c$ are either
trivial or of length two.

Suppose now $x_0\in V_1^c$.
\[
\xymatrix{
{x_0\in V_1^c}  \ar[r]^{\gamma^{-1}}
                \ar[ddr]^{\gamma{-1}} & {x_1\in V_1^c} \ar[r]^{\theta^{-1}} 
                                                      \ar[dr]^{\theta^{-1}} & {x_2\in M_c}\ar[r]^{\gamma} & {x_1}\ar[r]^{\theta}&{x_0}\\
& &{x_3\in M_t} \ar[r]^{\gamma}_=&{x_3}\ar[r]^{\theta}&{x_1}\\
&{x_4\in V_2^c} \ar[r]^{\theta^{-1}}
                \ar[dr]^{\theta^{-1}} & {x_5\in V_1^c} \ar[r]^{\gamma} & {x_4}\ar[r]^{\theta} &{x_0}\\
& & {x_6\in V_2^c} \ar[r]^{\gamma} & {x_4}\ar[r]^{\theta} & {x_0}\\
{x_1} \ar[r]^{\gamma^{-1}}&{x_7\in M_g} \ar[r]^{\theta^{-1}}_=&{x_7}\ar[r]^{\gamma}&{x_1}\ar[r]^{\theta} & {x_0}\\
}
\]
Thus, every orbit under the action of $\langle \omega\rangle$ which intersects $V_1^c$ is either trivial or of order two.

Suppose now $x_0\in V_1^t$.
\[
\xymatrix{
{x_0\in V_1^t}\ar[r]^{\gamma^{-1}}_=&{x_0}\ar[r]^{\theta^{-1}}
                                       \ar[dr]^{\theta^{-1}}&{x_1\in M_t}\ar[r]^{\gamma}_=&{x_1}\ar[r]^{\theta}&{x_0}\\
 & &{x_2\in M_c}\ar[r]^{\gamma} & {x_3\in V_1^g}\ar[r]^{\theta}_=&{x_3}\\
x_3\ar[r]^{\gamma^{-1}}&x_2\ar[r]^{\theta^{-1}}
                         \ar[dr]^{\theta^{-1}}&{x_4\in U_1^t}\ar[r]^{\gamma}_=&{x_4}\ar[r]^{\theta}&{x_2}\\
 & & {x_5\in U_1^c}\ar[r]^{\gamma}&{x_2}\ar[r]^{\theta}&{x_0}\\
{x_2}\ar[r]^{\gamma^{-1}}&{x_6\in U_1^g}\ar[r]^{\theta^{-1}}_=&{x_6}\ar[r]^{\gamma}&{x_2}\ar[r]^{\theta}&{x_0}
}
\]
In this case, it is possible for points in $V_1^t$ to be in orbits of
order one, two, or three under the action of $\langle\omega\rangle$.

Finally, suppose $x_0\in V_1^g$.
\[
\xymatrix{
x_0\in V_1^g\ar[r]^{\gamma^{-1}}
            \ar[dr]^{\gamma^{-1}} &{x_1\in M_g}\ar[r]^{\theta^{-1}}_=&{x_1}\ar[r]^{\gamma}&{x_0}\ar[r]^{\theta}_=&{x_0}\\
 & {x_2\in M_c}\ar[r]^{\theta^{-1}}
               \ar[dr]^{\theta^{-1}} & {x_3\in U_1^t}\ar[r]^{\gamma}_= & {x_3}\ar[r]^{\theta}&{x_2}\\
 & & {x_4\in U_1^c}\ar[r]^{\gamma}&{x_2}\ar[r]^{\!\!\!\!\!\!\!\!\!\!\!\!\!\!\!\!\!\!\!\!\!\!\!\!\!\!\!\!\theta}&{x_5=x_0\axby\in V_1^t}\\
{x_2}\ar[r]^{\gamma^{-1}}&{x_6\in U_1^g}\ar[r]^{\theta^{-1}}_=&{x_6}\ar[r]^{\gamma}&{x_2}\ar[r]^{\!\!\!\!\!\!\!\!\!\!\!\!\!\!\!\!\!\!\!\!\!\!\!\!\!\!\!\!\theta}&{x_5=x_0\axby\in V_1^t}\\
{x_5}\ar[r]^{\gamma^{-1}}_=&{x_5}\ar[r]^{\theta^{-1}}&{x_2}\ar[r]^{\gamma}&{x_0}\ar[r]^{\theta}_=&{x_0}
}
\]
Hence, under the action of $\langle \omega \rangle$, we see $V_1^g$ is
also a set where orbits in $\mathfrak{C}$ which interesect $V_1^g$ may
be of order one, two, or three.

We have now shown that if $p\in\mathfrak{C}$, then the cardinality of
$\mathcal{O}(p,\omega)$ is one, two, or three.  We conclude that the
order of $\omega$ divides six.
\end{proof}
We have therefore found a contradiction to the existence of our
supposed embedding of $Z^2*Z$ into $V$; the element $\omega$ is the
image of a non-trivial element of $Z^2*Z$ under an embedding, so $\omega$
must have infinite order.  In particular, there are no injections from
$Z^2*Z$ into $V$.